\documentclass[12pt]{amsart}
\usepackage{graphicx}
\usepackage[utf8]{inputenc} 
\usepackage[T1]{fontenc} 
\usepackage[english]{babel} 
\usepackage{mathrsfs} 
\usepackage{amssymb}
\usepackage[margin=1.4in]{geometry}

\theoremstyle{remark} 
\newtheorem*{remark}{Remark}

\theoremstyle{plain} 
\newtheorem{thm}{Theorem} 
\newtheorem{lem}[thm]{Lemma} 
\newtheorem{cor}[thm]{Corollary}
\newtheorem{prop}[thm]{Proposition}
\newtheorem{thmx}{Theorem}
\newtheorem{corx}{Corollary}

\newcommand{\K}{\mathbb{K}}
\newcommand{\Si}{\mathbb{S}}
\newcommand{\R}{\mathbb{R}}
\newcommand{\F}{\mathscr{F}}
\newcommand{\PP}{\mathscr{P}}
\newcommand{\E}{\mathcal{E}}
\DeclareMathOperator{\ess}{ess}
\DeclareMathOperator{\ind}{ind}

\DeclareMathOperator{\supp}{supp}
\DeclareMathOperator{\Tr}{Tr}
\DeclareMathOperator{\tr}{tr}
\DeclareMathOperator{\dom}{dom}
\DeclareMathOperator{\dist}{dist}
\DeclareMathOperator{\pv}{p.v.}
\DeclareMathOperator{\mre}{Re}
\begin{document} 
\title{The transmission problem on a three-dimensional wedge}
\date{\today} 

\author{Karl-Mikael Perfekt}
\address{Department of Mathematics and Statistics,
	University of Reading, Reading RG6 6AX, United Kingdom}
\email{k.perfekt@reading.ac.uk}

\begin{abstract}
We consider the transmission problem for the Laplace equation on an infinite three-dimensional wedge, determining the complex parameters for which the problem is well-posed, and characterizing the infinite multiplicity nature of the spectrum. This is carried out in two formulations leading to rather different spectral pictures. One formulation is in terms of square integrable boundary data, the other is in terms of finite energy solutions. We use the layer potential method, which requires the harmonic analysis of a non-commutative non-unimodular group associated with the wedge. 


\end{abstract}


\maketitle
\section{Introduction}

Let $\Gamma \subset \R^3$ be a surface, dividing $\R^3$ into interior and exterior domains $\Gamma_{+}$ and $\Gamma_{-}$, respectively. Given a spectral parameter $1 \neq \epsilon \in \mathbb{C}$ and boundary data $f$ and $g$ on $\Gamma$, the static transmission problem seeks a potential $U \colon \Gamma_+ \cup \Gamma_- \to \mathbb{C}$, harmonic in $\Gamma_+$ and $\Gamma_-$,
\begin{equation} \label{eq:trans1}
\Delta U = 0 \textrm { in } \Gamma_+\cup\Gamma_-,
\end{equation}
 such that
\begin{equation} \label{eq:trans2}
\Tr_+ U - \Tr_- U = f \quad \textrm{ and } \quad \partial_n^+ U - \epsilon \partial_n^- U = g \textrm{ on } \Gamma.
\end{equation}
Here $\Tr_{\pm} U$ and $\partial_n^{\pm} U$ denote the limiting boundary values and outward normal derivatives of $U$ on $\Gamma$, $+$ indicating an interior limiting approach, $-$ indicating exterior approach. For precise definitions, see equation~\eqref{eq:nttrace}. To discuss well-posedness, that is, existence and uniqueness of solutions, one has to impose growth and regularity conditions on the potential and the boundary data. We will consider two different sets of conditions which are widely used. One formulation is in terms of square integrable boundary data, the other in terms of finite energy potentials. We refer to these formulations as problems (L) and (E), respectively.

In elecrostatics, the parameter $\epsilon$ corresponds to the (relative) permittivity of a material and is a positive quantity, $\epsilon > 0$. In this case the problems (L) and (E) are very well-studied, and they have been shown to be well-posed for any Lipschitz surface $\Gamma$, see \cite{EFV92, EM04, FJR78, MW12, Ver84} and \cite{Cost07, CM85, HP13}, respectively. One approach to prove well-posedness is via the layer potential method, the success of which relies on the development and power of the theory of singular integrals.  By means of layer potentials, Problem (L) has even been shown to be well-posed for a wide class of very rough surfaces which are not Lipschitz regular \cite{HMT10}.

The transmission problem also appears as a quasi-static problem in electrodynamics, when an electromagnetic wave is scattered from an object that is much smaller than the wavelength. The permittivity $\epsilon$ is then complex and dependent on the frequency of the wave. In this setting, the properties of the transmission problem are very subtle. Problems (L) and (E) are no longer well-posed for certain $\epsilon \in \mathbb{C}$. When $\epsilon < 0$ this corresponds to the possibility of exciting surface plasmon resonances in nanoparticles made out of gold, silver, and other other materials \cite{AMRZ, ARYZ, Garnett, WKS08}. Metamaterials, specifically designed synthetic materials, can also exhibit effective permittivities with negative real part \cite{ASSE, Mil02, NMM93}. 

The set of $\epsilon \in \mathbb{C}$ for which the transmission problem is ill-posed -- the spectrum -- depends on the shape of the interface $\Gamma$. Strikingly, when the surface $\Gamma$ has singularities, the spectrum also depends heavily on the imposed growth and regularity conditions. For instance, when $\Gamma \subset \R^2$ is a curvilinear polygon in 2D, the spectrum of problem (L) is a union of two-dimensional regions in the complex plane, in addition to a set of real eigenvalues \cite{Mit02, Shele90}. On the other hand, the formulation of problem (E) is more directly grounded in physics. Accordingly, the spectrum of problem (E) is a real interval, plus eigenvalues, when $\Gamma$ is a curvilinear polygon \cite{BZ17, PP16}. In three dimensions, similar results hold for surfaces with rotationally symmetric conical points \cite{HP18}.

We will study the case when $\Gamma_{+} = \Gamma_{\alpha, +}$ is a three-dimensional infinite wedge with opening angle $\alpha \neq \pi$,
$$\Gamma_{\alpha, +} = \{(x\cos \theta, x\sin \theta,z) \in \R^3 \, : \, x > 0, \; 0 < \theta < \alpha\},$$
with boundary 
\begin{equation} \label{eq:wedgedecomp}
\Gamma_\alpha = \partial \Gamma_{\alpha, +} = \{(x,0,z) \in \R^3 \, : \, x \geq 0\} \cup \{(x\cos \alpha, x\sin \alpha,z) \in \R^3 \, : \, x \geq 0\}.
\end{equation}
By convention, we refer to $\Gamma_{\alpha, -} = \R^3 \setminus \overline{\Gamma_{\alpha, +}}$ as the exterior domain.
The two transmission problems (L) and (E) are given by
\begin{equation*}
\textrm{(L)} 	\begin{cases} 
\mathcal{M}^{+}(\nabla U), \mathcal{M}^{-}(\nabla U)  \in L^2(\Gamma_{\alpha}) \\
\Delta U = 0 \textrm{ in } \Gamma_{\alpha, +}\cup\Gamma_{\alpha, -}, \\
\Tr_+ U - \Tr_- U = f \in \dot{H}^1(\Gamma_{\alpha}), \\
\partial_n^+ U - \epsilon \partial_n^- U = g \in L^2(\Gamma_{\alpha}),
\end{cases} \textrm{(E)}\begin{cases} 
\int_{\R^3} |\nabla U|^2 \, dV < \infty, \\
\Delta U = 0 \textrm{ in } \Gamma_{\alpha, +}\cup\Gamma_{\alpha, -}, \\
\Tr_+ U - \Tr_- U = f \in \dot{H}^{1/2}(\Gamma_\alpha), \\
\partial_n^+ U - \epsilon \partial_n^- U = g \in \dot{H}^{-1/2}(\Gamma_{\alpha}).
\end{cases}
\end{equation*}
Here $\mathcal{M}^{\pm}(\nabla U)$ is an interior/exterior non-tangential maximal function of $\nabla U$, and $\dot{H}^s(\Gamma_\alpha)$ denotes a homogeneous Sobolev space of index $s$ along $\Gamma_\alpha$, see Sections~\ref{sec:L2} and \ref{sec:energy}. Alternatively, $\dot{H}^{1/2}(\Gamma_\alpha)$ can be viewed as the trace space of the space of harmonic functions in $\Gamma_{\alpha,+}$ satisfying that $\int_{\Gamma_{\alpha, +}} |\nabla U|^2 \, dV < \infty$. The negative space $\dot{H}^{-1/2}(\Gamma_\alpha)$ will be given an intrinsic description in terms of single layer potentials. The incomparability of $\dot{H}^{s_1}(\Gamma_\alpha)$ and $\dot{H}^{s_2}(\Gamma_\alpha)$, $s_1 \neq s_2$, will cause us some difficulties.

The purpose of this detailed study of the wedge is to have $\Gamma_{\alpha, +}$ serve as a model for general domains in $\R^3$ with edges. For domains with corners in 2D, the problems (L) and (E) are now well understood; a successful approach is to first consider the layer potential method on the infinite 2D-wedge \cite{GM90, KLY15, Lew90}, and to then reduce the study of curvilinear polygons to that of infinite wedges via a localization procedure. In 3D, similar approaches can be taken for domains with conical points \cite{HP18, KMR97, NQ12}. 

To fix the notation and to explain the layer potential approach at this point, we let $\K \colon L^2(\Gamma) \to L^2(\Gamma)$ denote the harmonic layer potential
$$\K f(r) = \pv \frac{1}{2\pi} \int_{\Gamma} \frac{(r - r') \cdot n(r)}{|r'-r|^3} f(r') \, d\sigma(r'), \quad r \in \Gamma, \; f \in L^2(\Gamma),$$
where $n(r)$ denotes the unit outward normal to $\Gamma$ at $r$, and $\sigma$ the surface measure on $\Gamma$. The adjoint $\K^\ast$ (with respect to $L^2(\Gamma)$) is known as the double layer potential or the Neumann--Poincar\'e operator. The single layer potential of a charge $f$ is given by
\begin{equation}
\label{eq:Sgendef}
\Si f(r) = \frac{1}{4\pi}\int_{\Gamma} \frac{1}{|r-r'|} f(r') \, d\sigma(r'), \quad r \in \R^3.
\end{equation}
When $\Gamma = \Gamma_\alpha$ we write $\K = \K_\alpha$ and $\Si = \Si_{\alpha}$.
Note that $\Si f$ is harmonic in $\Gamma_+ \cup \Gamma_-$. Differentiation leads to the jump formulas
$$\partial_n^{\pm} \Si f = \frac{1}{2}(\pm f - \K f) \textrm{ on } \Gamma.$$
The ansatz $U = \Si h_+$ in $\Gamma_+$ and $U = \Si h_-$ in $\Gamma_-$ hence relates the transmission problems (L) and (E) to spectral problems for the layer potential $\K$.

Previous studies of the transmission problem and layer potentials on the infinite three-dimensional wedge include the following. Eigensolutions to the transmission problem constructed via separation of variables can be found in \cite{DM72, WKS08}. Grachev and Maz'ya \cite{GM90} studied problem (E), using their results as a technical tool to describe the Fredholm radius of the double layer potential on certain weighted H\"older spaces for surfaces with edges. Fabes, Jodeit, and Lewis \cite{FJL77} observed, for $\alpha = \pi/2$, that the double layer potential $\K_{\alpha}^\ast$ on $\Gamma_{\alpha}$ can be regarded as a block matrix of convolution operators on the matrix group
$$G = \left \{(x,z) = \begin{pmatrix} 
x & z \\
0 & 1
\end{pmatrix}\, : \, x > 0, \; z \in \R\right \},$$ 
known as the $ax+b$ group. See also \cite{Qiao18}, where general angles and weighted $L^2$-spaces were considered. $G$ is a non-Abelian and non-unimodular group, and therefore does not support standard harmonic analysis. For $\alpha = \pi/2$, Fabes et. al. proved that $\K^\ast_{\alpha} \pm \mathbb{I}$ has an infinite-dimensional kernel on $L^{p}(\Gamma_\alpha)$ whenever $1 < p < 3/2$, where $\mathbb{I}$ denotes the identity operator. They proved this by constructing eigenfunctions through a rather delicate argument involving the partial Fourier transform in the $z$-variable. It is a natural idea to study layer potentials in the wedge by applying partial transforms in the $z$- and $x$-variables, cf. \cite{NL95, Schar04}, but such a procedure does not completely resolve $\K_{\alpha}^\ast$.

An explicit harmonic analysis for the $ax + b$ group was being developed around the same time that \cite{FJL77} was published, leading to the first example of a non-unimodular group equipped with a Plancherel theorem \cite{DM76, ET79, Khal74}. The corresponding Fourier transform of $G$ associates $\K_{\alpha}$ with four multiplication operators $M_T$, where $T \colon \mathcal{H} \to \mathcal{H}$ is an operator on an infinite-dimensional Hilbert space $\mathcal{H}$. As such, it does not provide a high level of resolution of the operator $\K_{\alpha}$, and it may seem that we are gaining an unmerited amount of information from the harmonic analysis of $G$. However, key to our results will be to identify each operator $T$ as a pseudo-differential operator of Mellin type \cite{Els87, Lew91, LP83}, after which we can apply the symbolic calculus of such operators to understand the spectrum of $\K_{\alpha}$.

Let $\Sigma_{\alpha} \subset \mathbb{C}$ denote the simple closed curve
$$\Sigma_{\alpha} = \left \{-\frac{\sin \left( (\frac{1}{2} +i\xi) (\pi - \alpha) \right)}{\sin \left( (\frac{1}{2} +i\xi) \pi \right)} \, : \, -\infty \leq \xi \leq \infty \right \},$$
and let $\widehat{\Sigma}_{\alpha}$ denote this curve together with its interior, see Figure~\ref{fig:spec}. For an operator $T \colon \mathcal{H} \to \mathcal{H}$, the spectrum $\sigma(T, \mathcal{H})$ is defined as usual, and we define the essential spectrum in the sense of Fredholm operators,
\begin{align*}
\sigma(T, \mathcal{H}) &= \{\lambda \in \mathbb{C} \, : \, T - \lambda \colon \mathcal{H} \to \mathcal{H} \textrm{ is not invertible}\}, \\
\sigma_{\ess}(T, \mathcal{H}) &= \{\lambda \in \mathbb{C} \, : \, T - \lambda \colon \mathcal{H} \to \mathcal{H} \textrm{ is not Fredholm}\}.
\end{align*}
In Theorem~\ref{thm:main1} we will characterize the spectrum of $\K_{\alpha} \colon L^{2,a}(\Gamma_{\alpha}) \to L^{2,a}(\Gamma_{\alpha})$ for $-1 < a < 3$, where 
\begin{equation} \label{eq:L2adef}
L^{2,a}(\Gamma_{\alpha}) = L^{2,a}(x^a \, dx \, dz) \oplus L^{2,a}(x^a \, dx \, dz),
\end{equation}
the orthogonal sum referring to the decomposition \eqref{eq:wedgedecomp} of $\Gamma_{\alpha}$. For simplicity, we shall only state the theorem for $a = 0$ here.
\begin{thmx} \label{thmx:A}
The spectrum of $\K_\alpha \colon L^{2}(\Gamma_\alpha) \to L^{2}(\Gamma_\alpha)$ satisfies that
$$\sigma(\K_\alpha, L^{2}(\Gamma_\alpha)) = \sigma_{\ess}(\K_\alpha, L^{2}(\Gamma_\alpha)) = -\widehat{\Sigma}_{\alpha} \cup \widehat{\Sigma}_{\alpha}.$$
Let $0 \neq \lambda \in -\widehat{\Sigma}_{\alpha} \cup \widehat{\Sigma}_{\alpha}$. Then,
\begin{enumerate}
\item $\lambda$ is an eigenvalue of the double layer potential $\K_\alpha^\ast \colon L^{2}(\Gamma_\alpha) \to L^{2}(\Gamma_\alpha)$ of infinite multiplicity, if $\lambda$ is an interior point of the spectrum;
\item there are generalized eigenfunctions of $\K_\alpha \colon L^{2}(\Gamma_\alpha) \to L^{2}(\Gamma_\alpha)$ corresponding to the point $\lambda$. In fact, $\lambda$ is an eigenvalue of $\K_\alpha \colon L^{2,2+\varepsilon}(\Gamma_{\alpha}) \to L^{2,2+\varepsilon}(\Gamma_{\alpha})$ of infinite multiplicity, for every $0 < \varepsilon < 1$, and, if $\lambda$ is an interior point, for $\varepsilon = 0$.
\end{enumerate}
Furthermore, $\K_\alpha \colon L^{2}(\Gamma_\alpha) \to L^{2}(\Gamma_\alpha)$ is normaloid, 
$$\|\K_\alpha\|_{B(L^{2}(\Gamma_\alpha))} = |\sigma(\K_\alpha, L^{2}(\Gamma_\alpha))| = \left|\sin\left( \frac{\pi - \alpha}{2} \right)\right|.$$
\end{thmx}
\begin{remark}
For an infinite 2D-wedge $\gamma_\alpha$ of angle $\alpha$, the spectrum of the double layer potential on $L^2(\gamma_\alpha)$ is the curve $-\Sigma_{\alpha} \cup \Sigma_{\alpha}$, without any interior \cite{Mit02}. In this case, neither the double layer potential, nor its adjoint, has any eigenvalues.
\end{remark}
 \begin{figure}[h!]
 	\centering
 	\includegraphics[width=0.80\linewidth]{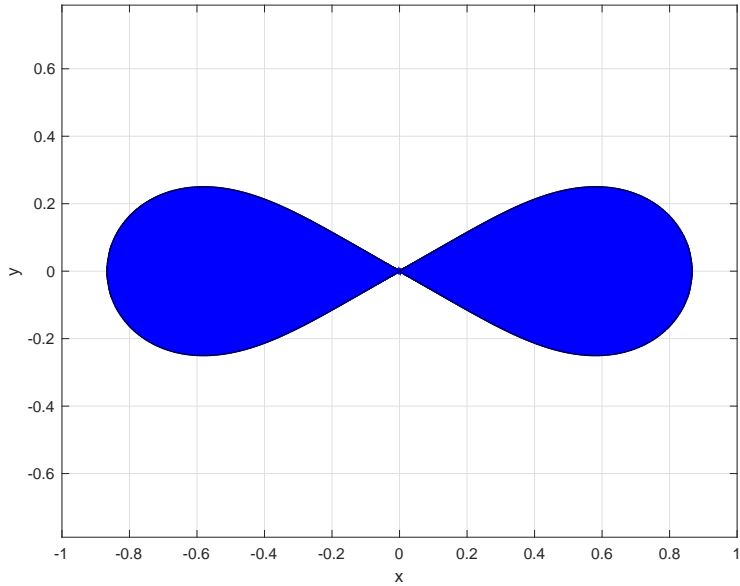}
 	\caption{Plot in the complex plane, $z = x+iy$, of the essential spectrum $-\widehat{\Sigma}_{\alpha} \cup \widehat{\Sigma}_{\alpha}$ of $\K_\alpha \colon L^{2}(\Gamma_\alpha) \to L^{2}(\Gamma_\alpha)$ for $\alpha = \pi/3$. Every point in the interior is an eigenvalue of infinite multiplicity of the double layer potential.}
 	\label{fig:spec}
 \end{figure}
In proving Theorem~\ref{thmx:A} we will show that any eigenvalue of $\K_\alpha \colon L^{2}(\Gamma_\alpha) \to L^{2}(\Gamma_\alpha)$ is real. Therefore, for non-real $\lambda$, the eigenfunctions of item (2) are truly generalized. Whether the same is valid for real $\lambda$ is left open. From Theorem~\ref{thmx:A} we obtain the promised corollary for the transmission problem (L).
\begin{corx}
	Let $1 \neq \epsilon \in \mathbb{C}$ and $f \in \dot{H}^1(\Gamma_{\alpha})$. Then the transmission problem {\rm (L)}
	is well posed (modulo constants) for all $g \in L^2(\Gamma_{\alpha})$ if and only if 
	$$
	\frac{1+\epsilon}{1-\epsilon} \notin -\widehat{\Sigma}_{\alpha} \cup \widehat{\Sigma}_{\alpha}.
	$$
\end{corx}

To treat problem (E), we follow Costabel \cite{Cost07} and Khavinson, Putinar, and Shapiro \cite{KPS07} by introducing the energy space $\E(\Gamma_\alpha)$ with norm
$$ \|f\|_{\E(\Gamma_\alpha)}^2 = \langle \Si_{\alpha} f, f \rangle_{L^2(\Gamma_\alpha)}.$$
This is motivated by Green's formula, which, ignoring technicalities, shows that $f \in \E(\Gamma_\alpha)$ if and only if $\int_{\R^3} |\nabla \Si_{\alpha} f|^2 \, dV < \infty$, see equation~\eqref{eq:energyproduct}.
Section~\ref{sec:energy} is devoted to proving that $\E(\Gamma_\alpha)$ coincides with the homogeneous Sobolev space $\dot{H}^{-1/2}(\Gamma_{\alpha})$, 
$$\E(\Gamma_\alpha) \simeq \dot{H}^{-1/2}(\Gamma_{\alpha}).$$
 The proof proceeds via interpolation, based on Dahlberg and Kenig's result \cite{DK87} that $\Si_{\alpha} \colon L^2(\Gamma_\alpha) \to \dot{H}^1(\Gamma_{\alpha})$ is an isomorphism, where $\Si_{\alpha}$ is understood as a map on the boundary $\Gamma_\alpha$.

The advantage of working with the energy space $\E(\Gamma_\alpha)$ is that $\K_{\alpha} \colon \E(\Gamma_\alpha) \to \E(\Gamma_\alpha)$ is self-adjoint, a consequence of the Plemelj formula
$$\Si_\alpha \K_\alpha = \K_\alpha^\ast \Si_\alpha,$$
which we will motivate in our setting. This explains why the energy formulation (E) of the transmission problem has a real spectrum. The study of the two operators $\K_{\alpha} \colon \E(\Gamma_\alpha) \to \E(\Gamma_\alpha)$ and $\K_{\alpha} \colon L^2(\Gamma_\alpha) \to L^2(\Gamma_\alpha)$ is reminiscent of   Krein's framework of symmetrizable operators \cite{Kre98}. However, a level of caution is necessary, since, unlike to Krein's setting, $\Si_{\alpha} \colon L^2(\Gamma_\alpha) \to L^2(\Gamma_\alpha)$ is an unbounded operator.

The main result concerning $\K_{\alpha} \colon \E(\Gamma_\alpha) \to \E(\Gamma_\alpha)$ is the following.
\begin{thmx} \label{thmx:B}
		The spectrum of the bounded self-adjoint operator $\K_\alpha \colon \E(\Gamma_\alpha) \to \E(\Gamma_\alpha)$ satisfies that
	$$\sigma(\K_\alpha, \E(\Gamma_\alpha)) = \sigma_{\textrm{ess}}(\K_\alpha, \E(\Gamma_\alpha)) = [-|1 - \alpha/\pi|, |1 - \alpha/\pi|].$$
	Every $0 \neq \lambda \in \sigma(\K_\alpha, \E(\Gamma_\alpha))$ is an eigenvalue of $\K_{\alpha} \colon L^{2,1+\varepsilon}(\Gamma_\alpha) \to L^{2,1+\varepsilon}(\Gamma_{\alpha})$ of infinite multiplicity, for $0 < \varepsilon < 2$.
\end{thmx}
\begin{remark}
	Eigensolutions to the transmission problem \eqref{eq:trans1}-\eqref{eq:trans2}, $f = g = 0$, are given in \cite{WKS08}, for permissible parameters $\epsilon$. These eigensolutions $U$ are constructed by separation of variables, and are thus periodic in $z$. Hence they could not satisfy that $\partial_n^\pm U \in L^{2,a}(\Gamma_\alpha)$ for any $a \in \R$. The relationship between the eigenfunctions of Theorems~\ref{thmx:A} and \ref{thmx:B} and the eigensolutions to the transmission problem is interesting, but unclear. The qualitative behavior of solutions to problem \eqref{eq:trans1}-\eqref{eq:trans2}, when $\mre \epsilon < 0$, is of importance to the study of plasmonics, as it is related to effects of field enhancement and confinement in plasmonic structures \cite{VS13, YA18}.
\end{remark}
Theorem~\ref{thmx:B} yields the expected corollary for the transmission problem. The sufficiency of the condition in Corollary~\ref{corx:B} has been shown previously in \cite[Theorem~1.6]{GM90}, but we will give a rather different proof.
\begin{corx} \label{corx:B}
	Let $1 \neq \epsilon \in \mathbb{C}$ and $f \in \dot{H}^{1/2}(\Gamma_\alpha)$. Then the transmission problem {\rm (E)}
	is well posed (modulo constants) for all $g \in \E(\Gamma_{\alpha}) \simeq \dot{H}^{-1/2}(\Gamma_{\alpha})$ if and only if 
	$$
	\frac{1+\epsilon}{1-\epsilon} \notin [-|1 - \alpha/\pi|, |1 - \alpha/\pi|].
	$$
\end{corx}

The paper is laid out as follows. In Section~\ref{sec:conv} we recall the convolution structure of $\K_{\alpha}$ and the harmonic analysis of the $ax+b$ group. Section~\ref{sec:L2} is devoted to proving Theorem~A. In Section~\ref{sec:energy} we identify the energy space $\E(\Gamma)$ with a homogeneous Sobolev space, and in Section~\ref{sec:Espec} we prove Theorem~B.

\subsection*{Acknowledgment} The author thanks Johan Helsing, Anders Karlsson, and Tobias Kuna for helpful discussions. The author is also grateful to the American Institute of Mathematics, San Jose, USA and the Erwin Schr\"odinger Institute, Vienna, Austria, in each of which some of this work was prepared.
\section{Convolution structure of layer potentials on the wedge} \label{sec:conv}
\subsection{Computations for the wedge}
Recall, for $0 < \alpha < 2\pi$, $\alpha \neq \pi$, that the wedge $\Gamma_{\alpha,+}$ has boundary
$$\Gamma_\alpha = \partial \Gamma_{\alpha, +} = \{(x,0,z) \in \R^3 \, : \, x \geq 0\} \cup \{(x\cos \alpha, x\sin \alpha,z) \in \R^3 \, : \, x \geq 0\}.$$
We write $L^{2,a}(dx \, dz) = L^{2}(\R_+ \times \R, x^a \, dx \, dz)$, so that
\begin{equation} \label{eq:orthdecomp}
L^{2,a}(\Gamma_{\alpha}) = L^{2,a}(dx \, dz) \oplus L^{2,a}(dx \, dz).
\end{equation}
The layer potential operator $\K_\alpha \colon  L^{2,a}(\Gamma_\alpha) \to L^{2,a}(\Gamma_\alpha)$ is, with respect to the orthogonal decomposition \eqref{eq:orthdecomp}, given by
\begin{equation} \label{eq:Kmatrixdecomp}
\K_\alpha = \begin{pmatrix} 
0 & K_\alpha \\
K_\alpha & 0 
\end{pmatrix},
\end{equation}
where, for appropriate functions $f \in L^{2,a}(dx \, dz)$ and $x > 0$, $z \in \R$, 
\begin{equation} \label{eq:Kformula}
K_\alpha f(x,z) = -\frac{1}{2\pi} \int_{-\infty}^{\infty} \int_0^\infty \frac{u\sin \alpha}{(x^2 - 2xu \cos \alpha +  u^2 + (z - v)^2)^{3/2}} f(u, v) \, du \, dv.
\end{equation}
As observed in \cite{FJL77, Qiao18}, through the change of variables
$$ \begin{cases}
u = x/s, \\
v = z - xt/s, \end{cases} \quad \frac{\partial(u,v)}{\partial(s,t)} = \frac{x^2}{s^3},$$
 we obtain that
\begin{equation} \label{eq:Kformula2}
K_\alpha f(x,z) = -\frac{\sin \alpha}{2\pi}  \int_{-\infty}^{\infty} \int_0^\infty \frac{1}{(1+ s^2 - 2s \cos \alpha + t^2)^{3/2}} f(x/s, z - xt/s) \frac{ds}{s} \, dt.
\end{equation}
It turns out that $\K_\alpha \colon L^{2,a}(\Gamma_\alpha) \to L^{2,a}(\Gamma_\alpha)$ is bounded for $-1 < a < 3$, see Lemma~\ref{lem:l2norm}. Thus, by duality, the double layer potential defines a bounded operator $\K_\alpha^\ast \colon L^{2,-a}(\Gamma_\alpha) \to L^{2,-a}(\Gamma_\alpha)$ for such $a$. Note here the convention of this paper; unless otherwise indicated, adjoint operations and dual spaces are calculated with respect to the inner product of $L^{2} = L^{2,0}$. 

In the present situation, as a map of functions on the unbounded graph $\Gamma_{\alpha}$, $\Si_\alpha \colon L^2(\Gamma_\alpha) \to L^2(\Gamma_\alpha)$ is not a bounded operator. However, it is densely defined, see Lemma~\ref{lem:Sdensedef}. In Lemma~\ref{lem:mapping} we will find that $\Si_\alpha$ can also be understood as a bounded map between certain weighted $L^p$-spaces.
As for $\K_{\alpha}$, the single layer potential can be formally written
$$\Si_\alpha = \begin{pmatrix} 
S_0 & S_\alpha \\
S_\alpha & S_0 
\end{pmatrix},$$
where
\begin{align} \label{eq:salphaformula}
S_\alpha f(x,z) &= \frac{1}{4\pi}  \int_{-\infty}^{\infty} \int_0^\infty \frac{1}{(x^2 - 2xu \cos \alpha  +  u^2 + (z - v)^2)^{1/2}} f(u, v) \, du \, dv \\ \nonumber
&= \frac{1}{4\pi} \int_{-\infty}^{\infty} \int_0^\infty \frac{x}{s(1+ s^2 - 2s \cos \alpha  + t^2)^{1/2}} f(x/s, z - xt/s)  \frac{ds}{s} \, dt.
\end{align}

\subsection{Convolution structure and harmonic analysis} \label{sec:group}
Consider the matrix group 
$$G = \left \{(x,z) = \begin{pmatrix} 
x & z \\
0 & 1
\end{pmatrix}\, : \, x > 0, \; z \in \R\right \},$$ 
in which multiplication corresponds to the composition of affine maps $w \mapsto xw + z$. That is,
$$(x,z)\cdot (s,t) = (xs, xt+z),$$
and 
$$(x,z) \cdot (s,t)^{-1} = (x,z) \cdot (1/s, -t/s) = (x/s, z - xt/s).$$
We always equip the group $G$ with its right Haar-measure $\frac{dx}{x} \, dz$. $G$ is a non-unimodular group; its left-invariant Haar measure is $\frac{dx}{x^2} \, dz$ and the Haar modulus is therefore $\Delta = \Delta(x,z) = x^{-1}$.

The connection between $G$ and $\K_\alpha$ is clear; $K_\alpha$ can be interpreted as a convolution operator, $K_\alpha f = f \star k_\alpha$, where
$$f \star g(x,z) = \int_{-\infty}^{\infty} \int_0^\infty f( (x,z) \cdot (s,t)^{-1}) g(s,t) \, \frac{ds}{s} \, dt.$$
Although we shall never make use of this, we point out that the convolution of $f$ and $g$ can also be computed with respect to the left structure of $G$, 
$$f \star g(x,z) = \int_{-\infty}^{\infty} \int_0^\infty f(s,t) g( (s,t)^{-1} \cdot (x,z) ) \, \frac{ds}{s^2} \, dt.$$
We will need Young's inequality for non-unimodular groups \cite[Lemma~2.1]{KR78}, stated for the right Haar measure.
\begin{lem} \label{lem:young}
Suppose that $1 \leq p,q,r \leq \infty$ satisfy $\frac{1}{p} + \frac{1}{q} = 1 + \frac{1}{r}$, and that $f \in L^p(G)$ and $g \in L^q(G)$. Then
$$\|f\Delta^{-\frac{1}{q'}} \star g\|_{L^r(G)} \leq \|f\|_{L^p(G)} \|g\|_{L^q(G)},$$
where $\frac{1}{q} + \frac{1}{q'} = 1.$
\end{lem}

The group $G$ was the first example of a non-unimodular group carrying a complete, explicit, harmonic analysis \cite{ET79, Khal74, KL72}. We shall now recall the main features. The reader should be warned that the statements below have been adapted to the right-invariant structure of $G$, while most of the references given treat the left structure. 

The construction is helped by the fact that $G = \R \rtimes \R_+$ is a semi-direct product of the two abelian groups $\R$ and $\R_+$, each of which comes with its own standard Fourier analysis. On $\R$ we have the usual Fourier transform $\mathcal{F}$, 
$$\mathcal{F}f(\xi) = \int_{-\infty}^\infty f(x) e^{-2\pi i x \xi} \, dx, \quad f\in L^1(\R), \; \xi \in \R,$$
which extends to a unitary map $\mathcal{F} \colon L^2(\R) \to L^2(\R)$. On $\R_+$, equipped with its Haar measure $\frac{dx}{x}$, the corresponding Fourier transform is known as the Mellin transform $\mathcal{M}$,
$$\mathcal{M}f(\xi) = \int_0^\infty f(x) x^{i\xi} \frac{dx}{x}, \quad f\in L^1(\R_+, \frac{dx}{x}), \; \xi \in \R.$$
Up to a constant scaling factor, $\mathcal{M}$ extends to a unitary $\mathcal{M} \colon L^2(\R_+) \to L^2(\R)$, where $L^2(\R_+) = L^2(\R_+, \frac{dx}{x})$.

The group $G$ has two infinite-dimensional irreducible unitary representations $\pi_\pm$ on $L^2(\R_+) = L^2(\R_+, \frac{dx}{x})$ \cite{GM47}, 
$$\pi_\pm(x,z)\eta(r) = e^{\mp 2\pi i z r} \eta(xr), \quad  \eta \in L^2(\R_+), \: r \in \R_+.$$
The unitary representations yield corresponding transforms $\F_{\pm}$. For $f \in L^1(G)$, $\F_{\pm}(f) \colon L^2(\R_+) \to L^2(\R_+)$ is the bounded operator given by
\begin{align*}
\F_\pm(f) \eta(r) &= \int_{-\infty}^\infty \int_0^\infty f(x, z) \pi_{\pm}(x,z) \eta(r)  \, \frac{dx}{x} \, dz \\
&= \int_{-\infty}^\infty \int_0^\infty e^{\mp 2\pi i z r} \eta(x r) f(x, z) \, \frac{dx}{x} \, dz, \quad \eta \in L^2(\R_+), \; f \in L^1(G).
\end{align*}
However, due to the non-unimodularity of $G$, it is not possible to immediately obtain a Plancherel theorem in terms of $\F_{\pm}$. In fact, there are compactly supported continuous $f$ for which $\F_\pm(f)$ is not even compact \cite{Khal74}. However, it is possible to obtain a Plancherel theorem by introducing an operator correction factor \cite{DM76, Fuhr06}.

In our case, the correction factor is given by $\delta$, where $\delta \eta(r) = \sqrt{r} \eta(r)$. Consider for $f \in L^2(G)$ the pair of operators $\PP_\pm(f) \colon L^2(\R_+) \to L^2(\R_+)$, formally given by $\PP_{\pm}(f) = \delta \F_\pm(f)$. More precisely,
$$\PP_\pm(f) \eta(r) = \sqrt{r}  \int_{-\infty}^\infty \int_0^\infty e^{\mp 2\pi i z r} \eta(x r) f(x, z) \frac{dx}{x} \, dz, \quad f \in L^2(G), \; \eta \in L^2(\R_+).$$
It is straightforward to verify that $\PP_{\pm}(f) \in \mathcal{S}_2$ for $f \in L^2(G)$, where $\mathcal{S}_2 = \mathcal{S}_2(L^2(\R_+))$ is the class of Hilbert-Schmidt operators on $L^2(\R_+)$. The ``Fourier transform'' of $L^2(G)$ is given by $\PP = (\PP_-, \PP_+)$, acting as a unitary map of $L^2(G)$ onto $\mathcal{S}_2^{(2)} = \mathcal{S}_2 \oplus \mathcal{S}_2$.
\begin{prop}[\cite{Khal74}]
The map $\PP \colon L^2(G) \to \mathcal{S}_2^{(2)}$ is onto and an isometry,
$$\|f\|_{L^2(G)}^2 = \|\PP_- f\|_{\mathcal{S}_2}^2 + \|\PP_+ f\|_{\mathcal{S}_2}^2.$$    
\end{prop}
Due to the correction factor, the convolution theorem is slightly asymmetrical.
\begin{prop}[\cite{Khal74}] \label{prop:conv}
If $k \in L^1(G)$ and $f \in L^2(G)$, then
$$\PP_\pm(f \star k) = \PP_{\pm}(f)\F_{\pm}(k).$$
\end{prop}

For $\gamma \in \R$, we let 
\begin{equation} \label{eq:Vdef}
V_{\gamma} f(x,z) = \Delta^{-\gamma} f(x,z) = x^{\gamma} f(x,z).
\end{equation}
Note that $f \in L^{2,a}(dx \, dz)$ if and only if $V_{\frac{a+1}{2}}f \in L^2(G)$, where $a\in\R$,  By the formula
$$\PP_{\pm}(f)\eta = \delta^{a+1}\PP_{\pm}(V_{\frac{a+1}{2}}f)[\delta^{-a-1}\eta],$$
valid at first for $f$ compactly supported in $G$, we can extend $\PP_{\pm}$ to $L^{2,a}(dx \, dz)$, in such a way that $\PP_{\pm}(f) \colon L^2(\R_+, r^{-a-2} \, dr) \to  L^2(\R_+, r^{-a-2} \, dr)$ is bounded when $f \in L^{2,a}(dx \, dz)$. Similarly, we interpret $\F_{\pm}(k)$ as a bounded operator on $L^2(\R_+, r^{-a-2} \, dr)$ for functions $k$ on $G$ for which $V_{\frac{a+1}{2}}k \in L^1(G)$. Note also that $f \star k \in L^{2,a}(dx \, dz)$ in this situation, by Young's inequality. For easy reference, we summarize what has been said in a lemma.
\begin{lem} \label{lem:convdensdef}
If $f \in L^{2,a}(dx \, dz)$ and $V_{\frac{a+1}{2}}k \in L^1(G)$, then 
$$\PP_{\pm}(f), \PP_{\pm}(f \star k), \F_{\pm}(k) \colon L^2(\R_+, r^{-a-2} \, dr) \to  L^2(\R_+, r^{-a-2} \, dr)$$
are bounded operators, and the convolution formula
 $$\PP_\pm(f \star k) = \PP_{\pm}(f)\F_{\pm}(k)$$
 is valid.
\end{lem}
\subsection{Multiplication operators}
By Proposition~\ref{prop:conv} we are led to consider multiplication operators on the Hilbert-Schmidt class $\mathcal{S}_2 = \mathcal{S}_2(\mathcal{H})$ of an infinite-dimensional Hilbert space $\mathcal{H}$ with norm $\|\cdot\|$. For a bounded operator $T \colon \mathcal{H} \to \mathcal{H}$ we denote by $M_T \colon \mathcal{S}_2 \to \mathcal{S}_2$ the operator of multiplication by $T$ on the right,
$$M_T S = ST, \quad S \in \mathcal{S}_2.$$
The following proposition is surely known.
\begin{prop} \label{prop:specmult}
We have that $\|M_T\|_{B(\mathcal{S}_2)} = \|T\|_{B(\mathcal{H})}$, $M_T^\ast = M_{T^\ast}$, and 
$$\sigma(M_T, \mathcal{S}_2) = \sigma_{\ess}(M_T, \mathcal{S}_2) = \sigma(T, \mathcal{H}).$$
Furthermore, if $\overline{\lambda}$ is an eigenvalue of $T^\ast$, then $\lambda$ is an eigenvalue of $M_T$ of infinite multiplicity.
\end{prop}
\begin{proof}
It is clear that $M_T^\ast = M_{T^\ast}$, since
$$\langle M_T S_1, S_2 \rangle_{\mathcal{S}_2} = \tr(S_1TS_2^\ast) = \tr(S_1(S_2T^\ast)^\ast) = \langle S_1, M_{T^\ast}S_2 \rangle_{\mathcal{S}_2}, \quad S_1,S_2 \in \mathcal{S}_2,$$
where $\tr$ denotes the usual trace of an operator in the trace class.

It is a standard fact that $\|ST\|_{\mathcal{S}_2} \leq \|S\|_{\mathcal{S}_2} \|T\|_{B(\mathcal{H})}$. Conversely, consider, for $g,h \in \mathcal{H}$, the rank-one operator $S_{g,h} = g \otimes h \in \mathcal{S}_2$,
$$S_{g,h} f = \langle f, g \rangle h, \quad  f \in \mathcal{H}.$$ Then $\|S_{g,h}\|_{\mathcal{S}_2} = \|g\| \|h\|$, while 
$$\|M_T S_{g,h}\|_{\mathcal{S}_2} = \|S_{T^\ast g,h}\|_{\mathcal{S}_2} = \|T^\ast g \| \|h\|.$$
It follows that $\|M_T\|_{B(\mathcal{S}_2)} = \|T\|_{B(\mathcal{H})}$. 

It is clear that $\sigma(M_T) \subset \sigma(T)$, for if $T - \lambda$ is invertible, then $M_{(T-\lambda)^{-1}}$ is the inverse of $M_T - \lambda$. 

If $\lambda \in \sigma(T)$ and $\overline{\lambda}$ is an eigenvalue of $T^\ast$ with non-zero eigenfunction $f$, then $\lambda$ is an eigenvalue of infinite multiplicity of $M_T$, since
\begin{equation}\label{eq:eiginfmult}
(M_T - \lambda) S_{f,h} = S_{(T^\ast - \overline{\lambda})f, h} = 0, \quad h \in \mathcal{H}.
\end{equation}

If $\lambda \in \sigma(T)$ and $T^\ast - \overline{\lambda}$ is injective but not bounded below, choose a sequence $(f_n) \subset \mathcal{H}$ such that $\|f_n\|=1$ for all $n$, but $\|T^\ast f_n - \overline{\lambda}f_n\| \to 0$ as $n \to \infty$. If $M_T - \lambda$ were Fredholm, then $M_T - \lambda \colon \mathcal{S}_2/J \to \mathcal{S}_2$ would be bounded below, where $J$ is the finite-dimensional kernel of $M_T - \lambda$. Since $\{S_{f_n,h} \, : \, h \in \mathcal{H}\}$ is an infinite-dimensional closed subspace of $\mathcal{S}_2$ we can for each $n$ pick $h_n$ with $\|h_n\| = 1$ such that $S_{f_n, h_n} \in J^\perp$ \cite[Lemma~2.3]{Schm12}. Then $\|S_{f_n, h_n}\|_{\mathcal{S}_2/J} = 1$, but
$$\|(M_T - \lambda) S_{f_n,h_n}\|_{\mathcal{S}_2} = \|S_{(T^\ast - \overline{\lambda})f_n, h_n}\|_{\mathcal{S}_2} = \|T^\ast f_n - \overline{\lambda}f_n\| \to 0, \quad n\to\infty,$$
a contradiction. Hence $M_T - \lambda$ is not Fredholm in this case either.

Finally, suppose that $\lambda \in \sigma(T)$ and that $T^\ast - \overline{\lambda}$ is bounded below but does not have full range. Then the range is not dense, and thus $\lambda$ is an eigenvalue of $T$. As in \eqref{eq:eiginfmult}, it follows that $(M_T - \lambda)^\ast = M_{T^\ast} - \overline{\lambda}$ has infinite-dimensional kernel. Hence $M_{T} - \lambda$ is not Fredholm. 

Adding up the different cases, we have shown that 
$$\sigma(T) \subset \sigma_{\ess}(M_T) \subset \sigma(M_T) \subset \sigma(T),$$
finishing the proof.
\end{proof}
\section{The $L^2$-spectrum} \label{sec:L2}
For $a \in \R$, recall the definition of $V_\gamma$ from \eqref{eq:Vdef} and note that
$$V_{\frac{a+1}{2}}  \colon L^{2,a}(dx \, dz) \to L^2(G)$$
 is unitary.  Hence $K_\alpha \colon L^{2,a}(dx \, dz) \to L^{2,a}(dx \, dz)$ is unitarily equivalent to 
 $$V_{\frac{a+1}{2}} K_\alpha V_{-\frac{a+1}{2}} \colon L^2(G) \to L^2(G).$$
 By equation \eqref{eq:Kformula2}, we see that
\begin{equation} \label{eq:Kunitary}
V_{\frac{a+1}{2}} K_\alpha V_{-\frac{a+1}{2}} f = f \star \Delta^{-\frac{a+1}{2}}k_\alpha, \quad f \in L^2(G),
\end{equation}
where
$$\Delta^{-\frac{a+1}{2}} k_\alpha(s,t) = -\frac{\sin \alpha}{2\pi}\frac{s^{\frac{a+1}{2}}}{(1+s^2 - 2s\cos \alpha + t^2)^{3/2}}.$$
The following lemma was first observed in \cite{FJL77, Qiao18}.
\begin{lem} \label{lem:l2norm}
For $-1 < a < 3$, $K_\alpha \colon L^{2,a}(dx \, dz) \to L^{2,a}(dx \, dz)$ is bounded with norm
$$\|K_\alpha\|_{B(L^{2,a}(dx \, dz))} \leq \left|\frac{\sin\left( (1-a)\frac{\pi - \alpha}{2} \right)}{\sin\left( (1-a) \frac{\pi}{2} \right)}\right|.$$
For $a=1$ the right-hand side should be interpreted as $|1 - \alpha/\pi|$.
\end{lem}
\begin{proof}
This follows by Young's inequality 
$$\|f \star g \|_{L^2(G)} \leq \|f\|_{L^2(G)} \|g\|_{L^1(G)},$$
and the computation
\begin{align*}
\|\Delta^{-\frac{a+1}{2}} k_\alpha\|_{L^1(G)} &= \frac{|\sin \alpha|}{2\pi} \int_0^\infty \int_{-\infty}^{\infty} \frac{s^{\frac{a+1}{2}}}{(1+s^2 - 2s\cos \alpha + t^2)^{3/2}} \, dt \frac{ds}{s} \\
&= \frac{|\sin \alpha|}{\pi} \int_0^\infty \frac{s^{\frac{a+1}{2}}}{1+s^2 - 2s\cos \alpha} \, \frac{ds}{s} = \left|\frac{\sin\left( (1-a)\frac{\pi-\alpha}{2} \right)}{\sin\left( (1-a) \frac{\pi}{2} \right)}\right|. \qedhere
\end{align*}
\end{proof}
For $-1 < a < 3$, let 
$$T_{\alpha, a}^\pm = \F_\pm(\Delta^{-\frac{a+1}{2}}k_\alpha),$$
and as in Proposition~\ref{prop:specmult}, let $M_{T_{\alpha, a}^\pm}$ denote the operator of right multiplication by $T_{\alpha, a}^\pm$ on $\mathcal{S}_2 = \mathcal{S}_2(L^2(\R_+))$. Then, by equation \eqref{eq:Kunitary} and Proposition~\ref{prop:conv}, $$K_\alpha \colon L^{2,a}(dx \, dz) \to L^{2,a}(dx \, dz)$$ is unitarily equivalent to 
$$\begin{pmatrix} 
M_{T_{\alpha,a}^{-}} & 0\\
0 & M_{T_{\alpha,a}^{+}}
\end{pmatrix} \colon \mathcal{S}_2^{(2)}(L^2(\R_+)) \to \mathcal{S}_2^{(2)}(L^2(\R_+)).$$
Explicitly, for $\eta \in L^2(\R_+)$ and $r > 0$,
$$
T_{\alpha,a}^{\pm} \eta (r) = \int_0^\infty \left( \frac{x}{r} \right)^{\frac{a+1}{2}} \int_{-\infty}^\infty e^{\mp 2\pi i z r} k_\alpha(x/r, z) \, dz \, \eta(x) \, \frac{dx}{x}.$$
Hence $T_{\alpha,a}^{\pm}$ is an integral operator given by
$$
T_{\alpha,a}^{\pm} \eta (r) = \int_0^\infty \left( \frac{x}{r} \right)^{\frac{a+1}{2}} T_\alpha^\pm(r,x)   \eta(x) \, \frac{dx}{x},$$
where
$$T_\alpha^\pm(r,x) := \int_{-\infty}^\infty e^{\mp 2\pi i z r} k_\alpha(x/r, z) \, dz = -2\sin(\alpha) r A_\alpha(r,x)^{-1} K_1(2\pi rA_\alpha(r,x)).$$
Here 
$$A_\alpha(r,x) = \left(1+\left(\frac{x}{r}\right)^2 - 2\frac{x}{r} \cos \alpha\right)^{1/2},$$
and $K_1$ is a modified Bessel function of the second kind \cite[p. 376]{AS72},
$$K_1(R) = R \int_{0}^\infty \frac{\cos t}{(R^2+t^2)^{3/2}} \, dt, \quad R > 0.$$
$K_1$ has the following asymptotics \cite[p. 378]{AS72},
\begin{equation} \label{eq:Bessel1}
K_1(R) = \frac{1}{R} + O(R), \quad R \to 0,
\end{equation}
and 
\begin{equation} \label{eq:Bessel2}
K_1(R) = \left(\frac{\pi}{2}\right)^{1/2} \frac{e^{-R}}{\sqrt{R}} \left( 1 + O \left(\frac{1}{R}\right) \right), \quad R \to \infty.
\end{equation}

\begin{lem} \label{lem:cpctperturb}
For $-1 < a < 3$, $T_{\alpha,a}^\pm \colon L^2(\R_+) \to L^2(\R_+)$ is a compact perturbation of the integral operator $I_{\alpha,a} \colon L^2(\R_+) \to L^2(\R_+)$ with kernel
$$\left(\frac{x}{r}\right)^{\frac{a+1}{2}}I_\alpha(r,x) := -\frac{\sin \alpha}{\pi}\chi_{(0,1)^2}(r,x) \left(\frac{x}{r}\right)^{\frac{a+1}{2}} A(r,x)^{-2},$$
where $\chi_{(0,1)^2}$ denotes the characteristic function of the square $(0,1)^2$.
\end{lem}
\begin{proof}
In fact,
$$\left(\frac{x}{r}\right)^{\frac{a+1}{2}}(T_\alpha^{\pm}(r,x) - I_\alpha(r,x)) \in L^2\left(\frac{dx}{x} \, \frac{dr}{r}\right),$$
so that $T_{\alpha,a}^{\pm} - I_{\alpha,a}$ is Hilbert-Schmidt. To see this, let
$$B(r,x) = r A_\alpha(r,x) = ((x - r \cos \alpha)^2 + r^2 \sin \alpha)^{1/2}.$$
$B(r,x)$ is bounded for $0 < r, x < 1$, so by \eqref{eq:Bessel1}
$$\int_0^1 \int_0^1 \left(\frac{x}{r}\right)^{a+1}|T_\alpha(r,x) - I_\alpha(r,x)|^2 \, \frac{dx}{x} \, \frac{dr}{r} \lesssim \int_0^1 \int_0^1 x^a r^{2-a} \, dx \, dr < \infty,$$
since $-1 < a < 3$. If $0 < r < 1$ and $1 < x < \infty$, then $B(r,x) \gtrsim x > 1$ and $A_\alpha(r,x) \gtrsim r^{-1}$, so by \eqref{eq:Bessel2} there is a constant $\gamma > 0$ such that
$$\int_0^1 \int_1^\infty \left(\frac{x}{r}\right)^{a+1}|T_\alpha(r,x)|^2 \, \frac{dx}{x} \, \frac{dr}{r} \lesssim \int_0^1 \int_1^\infty  \left(\frac{x}{r}\right)^{a+1} r^4 \frac{e^{-\gamma x}}{x} \, \frac{dx}{x} \, \frac{dr}{r} < \infty.$$
If $1 < r < \infty$ and $0 < x < 1$ we use that $B(r,x) \gtrsim r > 1$ and $A_\alpha(r,x) \gtrsim 1$, and therefore
$$\int_1^\infty \int_0^1 \left(\frac{x}{r}\right)^{a+1}|T_\alpha(r,x)|^2 \, \frac{dx}{x} \, \frac{dr}{r} \lesssim \int_1^\infty \int_0^1  x^a r^{-a} e^{-\gamma r} \, dx \, dr < \infty.$$
Finally, when $1 < r < \infty$ and $1 < x < \infty$ we have that $B(r,x) \gtrsim x + r$, and thus
\begin{equation*}
\int_1^\infty \int_1^\infty \left(\frac{x}{r}\right)^{a+1}|T_\alpha(r,x)|^2 \, \frac{dx}{x} \, \frac{dr}{r} \lesssim \int_1^\infty \int_1^\infty  x^a r^{-a} e^{-\gamma r}e^{-\gamma x} \, dx \, dr < \infty. \qedhere
\end{equation*}
\end{proof}
Observe that $I_{\alpha, a}$ is a truncated Mellin convolution operator (convolution on the group $\R_+$) with kernel
$$i_{\alpha,a}(s) = -\frac{\sin \alpha}{\pi}\frac{s^{\frac{3-a}{2}}}{1+s^2 -2s\cos \alpha},$$
in the sense that 
$$\left(\frac{x}{r}\right)^{\frac{a+1}{2}}I_\alpha(r,x) = i_\alpha(r/x), \quad 0 < r,x < 1.$$ 
For $-1 < a <3$, the kernel $i_{\alpha,a} \in L^1(\R_+)$ has Mellin transform
$$\mathcal{M}i_{\alpha,a}(\xi) = -\frac{\sin \left( (\frac{1-a}{2} +i\xi) (\pi - \alpha) \right)}{\sin \left( (\frac{1-a}{2} +i\xi) \pi \right)}, \quad \xi \in \R.$$
The range of this transform is the closed curve 
$$\Sigma_{\alpha,a} = \left \{-\frac{\sin \left( (\frac{1-a}{2} +i\xi) (\pi - \alpha) \right)}{\sin \left( (\frac{1-a}{2} +i\xi) \pi \right)} \, : \, -\infty \leq \xi \leq \infty \right \}.$$
For $a \neq 1$ this is a simple closed curve in $\mathbb{C}$, positively oriented if $-1 < a < 1$ and negatively oriented if $1 < a < 3$, in either case satisfying that $\Sigma_{\alpha, a} = \Sigma_{\alpha, 2-a}$. If $0 < \alpha < \pi$ then $\Sigma_{\alpha,a}$ lies in the left half-plane of $\mathbb{C}$, in the right half-plane if $\pi < \alpha < 2\pi$. For $a = 1$, $\Sigma_{\alpha,1}$ is the real interval between $0$ and $\alpha/\pi - 1$. It is clear that $\Sigma_{\alpha,a}$ is symmetric with respect to complex conjugation. The curves are increasing in $1 \leq a < 3$ in the sense that if $1 \leq a < a' < 3$, then every point of $\Sigma_{\alpha, a}$ but the origin is contained in the interior of $\Sigma_{\alpha, a'}$. For precise calculations we refer to \cite{Mit02}.

Lemma~\ref{lem:cpctperturb} shows that, with respect to the decomposition 
$$L^2(\R_+) = L^2((0,1), r^{-1}dr) \oplus L^2((1,\infty), r^{-1} dr),$$ we have that
$$T_{\alpha, a}^{\pm} = \begin{pmatrix} 
J_{\alpha,a} & *\\
* & *
\end{pmatrix},$$
where the entries marked $*$ are compact operators, and $J_{\alpha,a}$ is a pseudo-differential operator of Mellin type. There is a fully fledged theory of such operators developed by Elschner, Lewis, and Parenti \cite{Els87, Lew91, LP83}, together with a symbolic calculus which for $\lambda \notin \Sigma_{\alpha,a}$ gives the index of $J_{\alpha,a} - \lambda$, and thus of $T_{\alpha,a}^{\pm} - \lambda$, as the winding number $W(\Sigma_{\alpha,a}, \lambda)$ of $\lambda$ with respect to $\Sigma_{\alpha,a}$. In fact, the same operator $J_{\alpha,a}$ appears in computing the spectrum of double layer potentials on curvilinear polygons in 2D, and thus the relevant calculations already appear in \cite{Lew90, Mit02}. We do not give an account of the theory here, but instead summarize the conclusion it yields in the next proposition.
\begin{prop} \label{prop:Tspectrum}
The essential spectrum of $T_{\alpha,a}^{\pm}$ is 
$$\sigma_{\ess}(T_{\alpha,a}^{\pm}) = \Sigma_{\alpha,a}.$$
If $\lambda \notin \Sigma_{\alpha,a}$, then $T_{\alpha,a}^{\pm} - \lambda$ is Fredholm with index
$$\ind(T_{\alpha,a}^{\pm} - \lambda) = W(\Sigma_{\alpha,a}, \lambda).$$
\end{prop}

The classical Kellogg argument shows that any eigenvalue of $\K \colon L^2(\Gamma) \to L^2(\Gamma)$ must be real, in the case that $\Gamma$ is a bounded surface. However, this argument fails in the present setting, essentially because $L^2(\Gamma_\alpha)$ is not contained in the energy space $\E(\Gamma_\alpha)$, in the terminology of Section~\ref{sec:energy}. The next lemma offers a replacement of the Kellogg argument. For the statement, observe by \eqref{eq:Kformula} that $\K_\alpha \colon L^{2,1}(\Gamma_\alpha) \to L^{2,1}(\Gamma_\alpha)$ is a self-adjoint operator, hence has real spectrum.
\begin{lem} \label{lem:realeig}
If $-1 < a \leq 1$ and $\lambda \in \mathbb{C}$ is an eigenvalue of $\K_\alpha \colon L^{2,a}(\Gamma_\alpha) \to L^{2,a}(\Gamma_\alpha)$, or if $1 < a < 3$ and $\lambda$ is an eigenvalue of $\K_\alpha^\ast \colon L^{2,-a}(\Gamma_\alpha) \to L^{2,-a}(\Gamma_\alpha)$, then $\lambda \in \sigma(\K_\alpha, L^{2,1}(\Gamma_\alpha))$. In particular, $\lambda \in \R$.
\end{lem}
\begin{proof}
We give the argument for $-1 < a \leq 1$. The proof of the statement for $1 < a < 3$ is similar.
If $\lambda$ is an eigenvalue of $\K_\alpha \colon L^{2,a}(\Gamma_\alpha) \to L^{2,a}(\Gamma_\alpha)$, then, by \eqref{eq:Kmatrixdecomp}, either $\lambda$ or $-\lambda$ is an eigenvalue of $K_\alpha \colon L^{2,a}(dx \, dz) \to L^{2,a}(dx \, dz)$. Denote this latter eigenvalue by $\mu$.
 Let $f \in L^{2,a}(dx \, dz)$ be a non-zero eigenfunction and consider the decomposition
 $$f = f_1 + f_2, \quad f_1(x,z) = f(x,z)\chi_{(0,1)}(x), \; f_2(x,z) = f(x,z)\chi_{(1,\infty)}(x).$$
Noting that $a \leq 1$, we have that $f_1 \in L^{2,1}(dx \, dz)$, and therefore by Lemma~\ref{lem:l2norm} that $K_\alpha f_1 \in L^{2,1}(dx \, dz)$ as well. From the eigenvalue equation we hence obtain that
\begin{equation} \label{eq:eigvalue}
(K_\alpha - \mu) f_2 = -(K_\alpha - \mu) f_1 \in L^{2,1}(dx \, dz).
\end{equation}
In other words, $V_1(K_\alpha - \mu)V_{-1} V_1f_2 \in L^2(G)$, so that formal application of the Fourier transform yields
\begin{equation} \label{eq:eigvalue2}
\PP_{\pm}(V_1 f_2)(T_{\alpha, 1}^\pm - \mu) \in \mathcal{S}_2(L^2(\R_+)).
\end{equation}
To justify \eqref{eq:eigvalue2}, observe that $V_1f_2 \in L^{2, a-2}$ and that
$$V_{\frac{(a-2)+1}{2}}(\Delta^{-1}k_\alpha) = \Delta^{-\frac{a+1}{2}}k_\alpha \in L^1(G),$$
by the proof of Lemma~\ref{lem:l2norm}. Hence, by Lemma~\ref{lem:convdensdef}, the components of \eqref{eq:eigvalue2} are initially well-defined as bounded maps
$$\PP_{\pm}(V_1 f_2), T_{\alpha, 1}^\pm \colon L^2(\R_+, r^{-a} \, dr) \to L^2(\R_+, r^{-a} \, dr).$$
Equation~\eqref{eq:eigvalue} shows that $\PP_{\pm}(V_1 f_2)(T_{\alpha, 1}^\pm - \mu)$ in fact extends continuously to a Hilbert-Schmidt operator on $L^2(\R_+) = L^2(\R_+, r^{-1} \, dr)$.

Now, if $\lambda \notin \sigma(\K_\alpha, L^{2,1}(\Gamma_\alpha))$, then $K_\alpha - \mu \colon L^{2,1}(dx \, dz) \to L^{2,1}(dx \, dz)$ is invertible and thus $T_{\alpha, 1}^\pm - \mu \colon L^2(\R_+) \to L^2(\R_+)$ is invertible, by Propositions~\ref{prop:conv} and \ref{prop:specmult}. Therefore $\PP_{\pm}(V_1 f_2) \in \mathcal{S}_2$, that is, $V_1 f_2 \in L^2(G)$. Hence $f = f_1 + f_2 \in L^{2,1}(dx \, dz)$ and thus $\mu$ is an eigenvalue of $K_\alpha \colon L^{2,1}(dx \, dz) \to L^{2,1}(dx \, dz)$, from which it follows that $\lambda \in \sigma(\K_\alpha, L^{2,1}(\Gamma_\alpha))$, a contradiction.
\end{proof}
We are now ready to prove the main result of this section.  We denote by $\widehat{\Sigma}_{\alpha,a}$ the curve $\Sigma_{\alpha,a}$ together with its interior.
\begin{thm} \label{thm:main1}
For $-1 < a < 3$, the spectrum of $\K_\alpha \colon L^{2,a}(\Gamma_\alpha) \to L^{2,a}(\Gamma_\alpha)$ satisfies that
$$\sigma(\K_\alpha, L^{2,a}(\Gamma_\alpha)) = \sigma_{\ess}(\K_\alpha, L^{2,a}(\Gamma_\alpha)) = -\widehat{\Sigma}_{\alpha,a} \cup \widehat{\Sigma}_{\alpha,a}.$$
For $-1 < a < 1$, every point $\lambda$ in the interior of $-\widehat{\Sigma}_{\alpha,a} \cup \widehat{\Sigma}_{\alpha,a}$ is an eigenvalue of the Neumann--Poincar\'e operator $\K_\alpha^\ast \colon L^{2,-a}(\Gamma_\alpha) \to L^{2,-a}(\Gamma_\alpha)$ of infinite multiplicity. For $1 < a < 3$, every such point is an eigenvalue of $\K_\alpha \colon L^{2,a}(\Gamma_\alpha) \to L^{2,a}(\Gamma_\alpha)$ of infinite multiplicity. For $a=1$, $\K_\alpha \colon L^{2,1}(\Gamma_\alpha) \to L^{2,1}(\Gamma_\alpha)$ is self-adjoint and $\widehat{\Sigma}_{\alpha,a} = \Sigma_{\alpha,a}$ is a real interval,
$$\sigma(\K_\alpha, L^{2,1}(\Gamma_\alpha)) = \sigma_{\ess}(\K_\alpha, L^{2,1}(\Gamma_\alpha)) =[-|1 - \alpha/\pi|, |1 - \alpha/\pi|].$$
Furthermore, $\K_\alpha \colon L^{2,a}(\Gamma_\alpha) \to L^{2,a}(\Gamma_\alpha)$ is normaloid, 
$$\|\K_\alpha\|_{B(L^{2,a}(\Gamma_\alpha))} = |\sigma(\K_\alpha, L^{2,a}(\Gamma_\alpha))| = \left|\frac{\sin\left( (1-a)\frac{\pi - \alpha}{2} \right)}{\sin\left( (1-a) \frac{\pi}{2} \right)}\right|.$$
\end{thm}
\begin{proof}
By equation~\eqref{eq:Kmatrixdecomp}, $\K_\alpha - \lambda$ is invertible (Fredholm) on $L^{2,a}(\Gamma_\alpha)$ if and only if $K_\alpha - \lambda$ and $K_\alpha + \lambda$ are both invertible (Fredholm) on $L^{2,a}(dx \, dz)$.
Since, by Propositions~\ref{prop:specmult} and \ref{prop:Tspectrum},
$$\sigma_{\ess}(K_\alpha, L^{2,a}(dx \, dz)) = \sigma_{\ess}\begin{pmatrix} 
M_{T_{\alpha,a}^{-}} & 0\\
0 & M_{T_{\alpha,a}^{+}}
\end{pmatrix} = \sigma(T_{\alpha,a}^{-}) \cup \sigma(T_{\alpha,a}^+) \supset \widehat{\Sigma}_{\alpha,a},$$
we see that $-\widehat{\Sigma}_{\alpha,a} \cup \widehat{\Sigma}_{\alpha,a} \subset \sigma_{\ess}(\K_\alpha, L^{2,a}(\Gamma_\alpha))$. 

Suppose that $\lambda \in \widehat{\Sigma}_{\alpha,a} \setminus \Sigma_{\alpha,a}$. If $-1 < a < 1$, then 
$$\ind(T_{\alpha,a}^+ - \lambda) = \ind(T_{\alpha,a}^+ - \overline{\lambda}) = 1,$$
 so that $\overline{\lambda}$ is an eigenvalue of $T_{\alpha,a}^+$. Hence, by Proposition~\ref{prop:specmult}, $\lambda$ is an eigenvalue of infinite multiplicity of $M_{T_{\alpha,a}^+}^\ast = M_{(T_{\alpha,a}^+)^\ast}$, and thus also of $\K_\alpha^\ast \colon L^{2,-a}(\Gamma_\alpha) \to L^{2,-a}(\Gamma_\alpha)$. If instead $1 < a < 3$, then $\ind(T_{\alpha,a}^+ - \lambda) = -1$, and hence $\overline{\lambda}$ is an eigenvalue of $(T_{\alpha,a}^+)^\ast$. Again using Proposition~\ref{prop:specmult}, we conclude that $\lambda$ is an eigenvalue of infinite multiplicity of $\K_\alpha \colon L^{2,a}(\Gamma_\alpha) \to L^{2,a}(\Gamma_\alpha)$. The case when $\lambda$ lies in the interior of $-\widehat{\Sigma}_{\alpha,a}$ is analogous.

Finally, suppose that $\lambda \in \sigma(\K_\alpha, L^{2,a}(\Gamma_\alpha))$, but that $\lambda \notin -\widehat{\Sigma}_{\alpha,a} \cup \widehat{\Sigma}_{\alpha,a}$. Without loss of generality we may suppose that $\lambda \in \sigma(M_{T_{\alpha,a}^+}) = \sigma(T_{\alpha,a}^+)$. Then, since $\ind(T_{\alpha,a}^+ - \lambda) = 0$, $\lambda$ is an eigenvalue of $T_{\alpha,a}^+$ and $\overline{\lambda}$ an eigenvalue of $(T_{\alpha,a}^+)^\ast$. Thus, $\lambda$ is an eigenvalue of $M_{T_{\alpha,a}^+}$ and $\overline{\lambda}$ an eigenvalue of $M_{(T_{\alpha,a}^+)^\ast}$, immediately implying that $\lambda$ is an eigenvalue of $\K_\alpha \colon L^{2,a}(\Gamma_\alpha) \to L^{2,a}(\Gamma_\alpha)$ and $\overline{\lambda}$ an eigenvalue of $\K_\alpha^\ast \colon L^{2,-a}(\Gamma_\alpha) \to L^{2,-a}(\Gamma_\alpha)$. By Lemma~\ref{lem:realeig} we conclude that $\lambda \in \R$. On the other hand, by Lemma~\ref{lem:l2norm} we have that 
\begin{equation} \label{eq:lambdaineqpf}
|\lambda| \leq \|\K_\alpha\|_{B(L^{2,a}(\Gamma_{\alpha}))} = \|K_\alpha\|_{B(L^{2,a}(dx \, dz))} \leq \left|\frac{\sin\left( (1-a)\frac{\pi - \alpha}{2} \right)}{\sin\left( (1-a) \frac{\pi}{2} \right)}\right|.
\end{equation}
However, all real $\lambda$ satisfying \eqref{eq:lambdaineqpf} are contained in $-\widehat{\Sigma}_{\alpha,a} \cup \widehat{\Sigma}_{\alpha,a}$, a contradiction. The formula for the norm and spectral radius of $\K_\alpha \colon L^{2,a}(\Gamma_\alpha) \to L^{2,a}(\Gamma_\alpha)$ also follows from this last statement.
\end{proof}
\begin{remark}
By the symmetry $\Sigma_{\alpha, a} = \Sigma_{\alpha, 2-a}$ and the increasing nature of the curves $\Sigma_{\alpha, a}$, $1 \leq a < 3$, it is clear that Theorem~\ref{thm:main1} implies Theorem~\ref{thmx:A}.
\end{remark}
\begin{remark}
For an unbounded Lipschitz graph $\Gamma$, L. Escauriaza and M. Mitrea \cite{EM04} showed that $\sigma(\K, L^2(\Gamma))$ is contained in a certain hyperbola which only depends on the Lipschitz character of $\Gamma$. Perhaps unsurprisingly, Theorem~\ref{thm:main1} shows that their result is sharp for the wedge boundaries $\Gamma_{\alpha}$.
\end{remark}
To give the application of Theorem~\ref{thm:main1} to the transmission problem we need to recall some of the layer potential theory of the unbounded Lipschitz graph $\Gamma_\alpha$ \cite{DK87}. For a function $V$ on $\Gamma_{\alpha,+}$ or $\Gamma_{\alpha,-}$, the non-tangential maximal function $\mathcal{M}^{\pm}V \colon \Gamma_\alpha \to [0,\infty]$ is given by
$$\mathcal{M}^\pm (V)(r) = \sup\{|V(r')| \, : \, r' \in \Gamma_{\alpha,\pm}, \; |r'-r| \leq 2\dist(r',\Gamma_{\alpha})\}, \quad r \in \Gamma_{\alpha}.$$
Consider the two spaces of harmonic functions
$$\dot{h}^1(\Gamma_{\alpha, \pm}) = \left\{\Delta U = 0 \textrm{ in } \Gamma_{\alpha, \pm} \, : \, \mathcal{M}^\pm (\nabla U) \in L^2(\Gamma_{\alpha})\right\},$$
which we implicitly consider as quotient spaces over the constant functions. Every $U \in \dot{h}^1(\Gamma_{\alpha, \pm})$ has non-tangential boundary values and outward normal derivatives pointwise almost everywhere on $\Gamma_{\alpha}$. That is,
\begin{equation} \label{eq:nttrace}
\Tr_\pm U(r) = \textrm{n.t.-}\lim_{r' \to r} U(r'), \quad \partial_n^\pm U(r) = \textrm{n.t.-}\lim_{r' \to r} n(r) \cdot \nabla U(r')
\end{equation}
exist for almost all $r \in \Gamma_{\alpha}$, where the convergence takes place in all non-tangential regions 
$$\{r' \in \Gamma_{\alpha,+} \, : \, |r' - r| \leq (1+c)\dist(r',\Gamma_{\alpha}) \}, \quad c > 0.$$
 Furthermore, $\partial_n^{\pm} U \in L^2(\Gamma_{\alpha})$, and $\Tr_\pm U$ belongs to the homogeneous Sobolev space $\dot{H}^1(\Gamma_{\alpha})$ on $\Gamma_\alpha$, consisting of the functions $f \in L^2_{\textrm{loc}}(\Gamma_\alpha)$ such that the (tangential) gradient of $f$ on $\Gamma_{\alpha}$ belongs to $L^2(\Gamma_{\alpha})$. A more precise definition of $\dot{H}^1(\Gamma_{\alpha})$ is given in Section~\ref{sec:energy}. $\dot{H}^1(\Gamma_{\alpha})$ is a Hilbert space modulo constants. Hence $\Tr_{\pm} \colon \dot{h}^1(\Gamma_{\alpha, \pm}) \to \dot{H}^1(\Gamma_{\alpha})$ and $\partial_n^\pm \colon \dot{h}^1(\Gamma_{\alpha, \pm}) \to L^2(\Gamma_{\alpha})$ are bounded operators.

The single layer potential is an isomorphism as a map $\Si_{\alpha} \colon L^2(\Gamma_{\alpha}) \to \dot{H}^1(\Gamma_{\alpha})$ \cite[Lemma 3.1]{DK87}. Evaluating $\Si_{\alpha} f$ instead on either $\Gamma_{\alpha, +}$ or $\Gamma_{\alpha, -}$, see equation~\eqref{eq:Sgendef}, yields isomorphisms $\Si_{\alpha} \colon L^2(\Gamma_{\alpha}) \to \dot{h}^1(\Gamma_{\alpha, \pm})$. Furthermore, by the weak singularity of the kernel, we have that
$$\Tr_{\pm} \Si_{\alpha} f = \Si_{\alpha} f \in \dot{H}^1(\Gamma_{\alpha}), \quad f \in L^2(\Gamma_{\alpha}).$$
In other words, the interior Dirichlet problem
$$\begin{cases}
\mathcal{M}^{+}(\nabla U) \in L^2(\Gamma_{\alpha}), \\
\Delta U = 0 \textrm{ in } \Gamma_{\alpha, +}, \\ 
\Tr_{+} U = g \in \dot{H}^1(\Gamma_{\alpha})
\end{cases}$$
is well posed (modulo constants), and the solution is of the form of a single layer potential, $U = \Si_{\alpha} f$, $f \in L^2(\Gamma_{\alpha})$. The same statement holds for the exterior Dirichlet problem.

To treat the transmission problem we make use of the jump formulas \cite[p. 149]{EM04}
\begin{equation} \label{eq:L2jumpformulas}
\partial_n^{+} \Si_\alpha f = \frac{1}{2}(f - \K_\alpha f), \quad \partial_n^{-} \Si_\alpha f(r) = \frac{1}{2}(-f - \K_\alpha f), \quad f \in L^2(\Gamma_{\alpha}).
\end{equation}

\begin{cor}
	Let $1 \neq \epsilon \in \mathbb{C}$ and $f \in \dot{H}^1(\Gamma_{\alpha})$. Then the transmission problem 
	$$\begin{cases} 
	\mathcal{M}^{+}(\nabla U), \mathcal{M}^{-}(\nabla U)  \in L^2(\Gamma_{\alpha}) \\
	\Delta U = 0 \textrm{ in } \Gamma_{\alpha, +}\cup\Gamma_{\alpha, -}, \\
	\Tr_+ U - \Tr_- U = f \in \dot{H}^1(\Gamma_{\alpha}), \\
	\partial_n^+ U - \epsilon \partial_n^- U = g \in L^2(\Gamma_{\alpha})
	\end{cases}$$
	is well posed (modulo constants) for all $g \in L^2(\Gamma_{\alpha})$ if and only if 
	\begin{equation} \label{eq:L2solvcond}
	\frac{1+\epsilon}{1-\epsilon} \notin -\widehat{\Sigma}_{\alpha,0} \cup \widehat{\Sigma}_{\alpha,0}.
	\end{equation}
\end{cor}
\begin{proof}
	By well-posedness of the Dirichlet problems there are densities $h_\pm \in L^2(\Gamma_{\alpha})$ and a constant $c$ such that $U = \Si_{\alpha} h_+ + c$ in $\Gamma_{\alpha, +}$ and $U = \Si_{\alpha} h_- + c$ in $\Gamma_{\alpha, -}$. By the jump formulas \eqref{eq:L2jumpformulas}, the transmission problem is then equivalent to the system
	$$\begin{cases}
	\Si_\alpha(h_+ - h_-) = f, \\
	\left(\K_{\alpha} - \frac{1+\epsilon}{1-\epsilon}\mathbb{I}\right) h_+ = -\frac{1}{1-\epsilon} \left[2g + \epsilon(\K_\alpha+\mathbb{I})(h_+ - h_-) \right]
	\end{cases}$$
	on $\Gamma_{\alpha}$, where $\mathbb{I}$ denotes the identity map. This system is uniquely solvable if and only if \eqref{eq:L2solvcond} holds, by Theorem~\ref{thm:main1} and the fact that $\Si_{\alpha} \colon L^2(\Gamma_{\alpha}) \to \dot{H}^1(\Gamma_{\alpha})$ is an isomorphism.
\end{proof}
\section{The energy space on unbounded Lipschitz graphs} \label{sec:energy}
\subsection{Identification with a fractional homogeneous Sobolev space}
In this section only, we will consider the more general situation where $\Gamma$ is an unbounded Lipschitz graph,
$$\Gamma = \{r = (x,y,z) \in \R^3 \, : \, z = \varphi(x,y) \},$$
where $\varphi \colon \R^2 \to \R$ is Lipschitz continuous. We think of the region above $\Gamma$ as the interior domain $\Gamma_+$, the region below it as the exterior $\Gamma_-$. The energy space $\E(\Gamma)$ in the case when $\Gamma$ is an infinite cone was important in \cite{HP18}, but was not shown to coincide with a Sobolev space. We therefore prove this identification for general Lipschitz graphs here. The considerations of this section apply equally well to the case of an unbounded Lipschitz graph embedded in $\R^n$, $n \geq 3$, but we restrict ourselves to $n=3$ for simplicity of notation.

Denote the space of compactly supported functions $f \in L^2(\Gamma)$ by $L^2_c(\Gamma)$. Then 
\begin{equation} \label{eq:energyproduct}
\langle \Si f, g \rangle_{L^2(\Gamma)} = \int_{\R^3} \nabla \Si f \cdot \overline{\nabla \Si g} \, dV, \quad f,g \in L^2_c(\Gamma).
\end{equation}
This is a standard identity which follows from Green's formula and the jump formulas \eqref{eq:L2jumpformulas} for the interior and exterior normal derivatives of $\Si f$ on $\Gamma$. When $\Gamma$ is smooth, bounded, and connected, equation \eqref{eq:energyproduct} may be found in \cite[Lemma 1]{KPS07}. The approximation procedure of \cite[Theorem~1.12]{Ver84} extends it to connected bounded Lipschitz surfaces, see for example \cite{PP14}. Finally, exhausting $\Gamma_+$ with a suitable increasing sequence of bounded Lipschitz domains yields \eqref{eq:energyproduct} for unbounded Lipschitz graphs. See Lemma~\ref{lem:plemelj} for a similar argument spelled out in greater detail.

Consider the inner product
\begin{equation} \label{eq:energyinner}
\langle f, g \rangle_{\E(\Gamma)} = \langle \Si f, g \rangle_{L^2(\Gamma)},
\end{equation}
initially for functions $f,g \in L^2_c(\Gamma)$.
Equation \eqref{eq:energyproduct} shows positive definiteness; if $\|f\|_{\E(\Gamma)}^2 = \langle f, f \rangle_{\E(\Gamma)} = 0$, then $\nabla \Si f(r) = 0$ for $r \in \R^3 \setminus \Gamma_\alpha$. However, this implies that $\partial_n^{+} \Si f = 0$, see equation \eqref{eq:nttrace}, which, unless $f=0$, is incompatible with the estimate
$$\|\partial_n^{+} \Si f\|_{L^2(\Gamma)} \gtrsim \|f\|_{L^2(\Gamma)}$$
from \cite{Ken85}. We define the energy space $\E(\Gamma)$ as the completion of $L^2_c(\Gamma)$ under this inner product.

When $\Gamma$ is a connected bounded Lipschitz surface, the energy space $\E(\Gamma)$ consists precisely of the distributions $f$ on $\Gamma$ in the inhomogeneous Sobolev space $H^{-1/2}(\Gamma)$ \cite{Cost07}. We will show that for an unbounded Lipschitz graph $\Gamma$ this remains true upon replacing $H^{-1/2}(\Gamma)$ by a homogeneous Sobolev space. 

Let $\mathcal{F} \colon L^2(\R^2) \to L^2(\R^2)$ denote the usual two-dimensional Fourier transform. For $0 \leq s \leq 1$, we define the homogeneous Sobolev space $\dot{H}^{s}(\R^2)$ as the completion of $C^\infty_c(\R^2)$ under the norm
\begin{equation} \label{eq:sobolevr2norm}
\|f\|_{\dot{H}^{s}(\R^2)}^2 = \int_{\R^2}|\mathcal{F}f(\xi)|^2 |\xi|^{2s} d\xi.
\end{equation}
We refer to \cite[Ch.~1]{BCD11} for the basics of homogeneous Sobolev spaces.
When $0 < s < 1$, the norm can also be computed as a Slobodeckij norm, see for example \cite[Proposition 3.4]{NEPV12},
$$\|f\|_{\dot{H}^{s}(\R^2)}^2 = c_s \int_{\R^2} \int_{\R^2}  \frac{|f(r) - f(r')|^2}{|r-r'|^{2(1+s)}} \, dr \, dr', \quad 0 < s < 1, \; f\in C^\infty_c(\R^2),$$
where $c_s$ is a constant depending on $s$.
For $0 \leq s < 1$, we emphasize that the completion  $\dot{H}^{s}(\R^2)$ is a space of functions. In fact, there is an injective embedding of $\dot{H}^{s}(\R^2)$ into $L^{2/(1-s)}(\R^2)$ \cite[Theorem~2.1]{EL12}. For $s=1$, $\dot{H}^{1}(\R^2)$ is the quotient of a semi-Hilbert space of functions with the subspace of constant functions. More precisely, $\dot{H}^{1}(\R^2)$ is the Hilbert space of $L^2_{\textrm{loc}}(\R^2)$-functions $f$ modulo constants such that $\nabla f \in L^2(\R^2)$. We define the negative index spaces $\dot{H}^{-s}(\R^2)$ as the dual spaces of $\dot{H}^{s}(\R^2)$ with respect to the $L^2(\R^2)$-pairing. Note that \eqref{eq:sobolevr2norm} remains valid for $-1 \leq s < 0$, in the sense that the Fourier transform extends to a unitary
\begin{equation} \label{eq:sobolevr2norm2}
\mathcal{F} \colon \dot{H}^s(\R^2) \to L^2(\R^2, |\xi|^{2s} \, d\xi), \quad -1 \leq s \leq 1.
\end{equation}

Alternatively, homogeneous Sobolev spaces may be understood in terms of the Riesz potential \cite[Section~3]{NEPV12}. For $0 < s \leq 1$, the Riesz potential is given by
\begin{equation} \label{eq:rieszpotdef}
\mathcal{I}_s f(r) = (-\Delta)^{s/2}f(r) = \mathcal{F}^{-1}(|\xi|^{-s}\mathcal{F}f)(r) = c_s' \int_{\R^2} \frac{f(r')}{|r-r'|^{2-s}} \, dr', \quad r \in \R^2.
\end{equation}
where $c_s'$ is a constant depending on $s$. Clearly, $\mathcal{I}_s \colon L^2(\R^2) \to \dot{H}^{s}(\R^2)$ is a unitary map, and by duality, so is $\mathcal{I}_s \colon \dot{H}^{-s}(\R^2) \to L^2(\R^2)$. 

We naturally interpet functions $f$ on $\Gamma$ as functions on $\R^2$, by letting 
$$\Lambda f(x,y) = f(x,y,\varphi(x,y)), \quad(x,y) \in \R^2.$$ For $0 \leq s \leq 1$, we let $\dot{H}^{s}(\Gamma) = \Lambda^{-1}\dot{H}^{s}(\R^2)$, in the sense that $\dot{H}^{s}(\Gamma)$ is the completion of $\Lambda^{-1}C_c^\infty(\R^2)$ under the norm $\|f\|_{\dot{H}^{s}(\Gamma)} = \|\Lambda f\|_{\dot{H}^{s}(\R^2)}$. We define $\dot{H}^{-s}(\Gamma)$ as the dual of $\dot{H}^{s}(\Gamma)$ with respect to the $L^2(\Gamma)$-pairing. 
	
\begin{lem} \label{lem:l2cdense}
Every function $g \in L^2_c(\Gamma)$ induces a distinct element $\ell_g \in \dot{H}^{-1/2}(\Gamma)$,
$$\ell_g(f) = \langle f, g \rangle_{L^2(\Gamma)}, \quad f \in \Lambda^{-1}C_c^\infty(\R^2),$$
and 
\begin{equation} \label{eq:normgamma}
\|\ell_g\|_{\dot{H}^{-1/2}(\Gamma)} = \|\rho \Lambda g\|_{\dot{H}^{-1/2}(\R^2)}, \quad g \in L^2_c(\Gamma),
\end{equation}
where
\begin{equation} \label{eq:rhodef}
\rho(x,y) = (1 + |\nabla \varphi(x,y)|^2)^{1/2}.
\end{equation}
The space of all such functionals is dense in $\dot{H}^{-1/2}(\Gamma)$.
\end{lem}
\begin{proof}
Note that $d\sigma(x,y) = \rho(x,y) \, dx \, dy$. Since $\dot{H}^{1/2}(\R^2) \subset L^4(\R^2)$, we deduce that $L^2_c(\R^2) \subset \dot{H}^{-1/2}(\R^2)$. Therefore $\ell_g$ induces a bounded functional on $\dot{H}^{1/2}(\Gamma)$, since $\rho \Lambda g \in L^2_c(\R^2)$ and 
$$\ell_g(f) = \langle \Lambda f, \rho \Lambda g \rangle_{L^2(\R^2)}.$$
This last formula also implies \eqref{eq:normgamma}. It is clear that $\ell_{g_1} = \ell_{g_2}$ if and only if $g_1 = g_2$ almost everywhere. The density follows from the fact that the elements of $\dot{H}^{1/2}(\R^2)$ are functions. 
\end{proof}
We interpret Lemma~\ref{lem:l2cdense} by saying that $L^2_c(\Gamma)$ is densely contained in $\dot{H}^{-1/2}(\Gamma)$, and we do not notationally distinguish between $\ell_g$ and $g$ from this point on.
 
 By the group property 
 $$\mathcal{I}_{s_1}\mathcal{I}_{s_2} = \mathcal{I}_{s_1+s_2}, \quad 0 < s_1, s_2 < 1,$$
 and the unitarity of $\mathcal{I}_{1/2} \colon \dot{H}^{-1/2}(\R^2) \to L^2(\R^2)$, we find that
 \begin{equation} \label{eq:energyspaceriesz}
 \langle \mathcal{I}_1 f, f \rangle_{L^2(\R^2)} = \|f\|_{\dot{H}^{-1/2}(\R^2)}^2, \quad f \in L^2_c(\R^2).
 \end{equation}
Furthermore, if $f \in L^2_c(\Gamma)$ is a nonnegative function, then $\mathcal{I}_s \rho \Lambda f \simeq \Lambda \Si f$, since the kernels of $\mathcal{I}_s \rho \Lambda$ and $\Lambda \Si$ are comparable. Comparing \eqref{eq:energyinner}, \eqref{eq:normgamma}, and \eqref{eq:energyspaceriesz} thus yields that 
  $$\|f\|_{\E(\Gamma)} \simeq \|f\|_{\dot{H}^{-1/2}(\Gamma)}, \quad 0 \leq f \in L^2_c(\Gamma).$$
To extend this estimate to general functions, we appeal to an interpolation argument, beginning with the following lemma.
\begin{lem} \label{lem:Sdensedef}
The space
$$L^2_{c,0}(\Gamma) = \left\{f \in L^2_c(\Gamma) \, : \, \int_{\Gamma} f \, d\sigma = 0\right\}$$
is contained and dense in $\dot{H}^{-1}(\Gamma)$, $\dot{H}^{-1/2}(\Gamma)$, $\E(\Gamma)$, and $L^2(\Gamma)$. Furthermore, $\Si$ maps $L^2_{c,0}(\Gamma)$ into $L^2(\Gamma)$, and if $f \in L^2_{c,0}(\Gamma)$, then $\Si f (r) = O(|r|^{-2})$ as $r \to\infty$.
\end{lem}
\begin{proof}
A direct proof that $L^2_{c,0}(\Gamma) \subset  \dot{H}^{-1}(\Gamma)$ goes as follows.
Suppose that $f \in L^2_{c,0}(\Gamma)$, and let $\rho$ be as in \eqref{eq:rhodef}. Then $\mathcal{F}(\rho \Lambda f)$ is real analytic on $\R^2$ and
$$\mathcal{F}(\rho \Lambda f)(0) = \int_{\Gamma} f \, d\sigma = 0.$$
Therefore $\mathcal{F}(\rho \Lambda f) \in L^2(\R^2, |\xi|^{-2} \, d\xi)$, from which it follows that $f \in \dot{H}^{-1}(\Gamma)$. That is, $g \mapsto \langle g, f \rangle_{L^2(\Gamma)}$ defines a continuous functional on $\dot{H}^{1}(\Gamma)$. The density of $L^2_{c,0}(\Gamma)$ in $\dot{H}^{-1}(\Gamma)$ is immediate from the fact that the elements of $\dot{H}^1(\R^2)$ are $L^2_{\textrm{loc}}(\R^2)$-functions modulo constants.

Next, let $g_n \in L^2_c(\Gamma)$ be defined by
$$g_n(x,y, \varphi(x,y)) = \frac{1}{n^2}\chi_{(-n/2,n/2)^2}(x,y) \rho(x,y)^{-1}.$$
Then $\int_{\Gamma} g_n \, d\sigma = 1$, and
$$0 \leq \Si g_n(r) \lesssim \frac{1}{n}, \quad r \in \Gamma,$$
by a straightforward estimate.
Hence 
$$\|g_n\|_{\E(\Gamma)} = \sqrt{\langle \Si g_n, g_n \rangle_{L^2(\Gamma)}} \lesssim \frac{1}{\sqrt{n}}.$$
Similarly we see from \eqref{eq:energyspaceriesz} that 
$$\|g_n\|_{\dot{H}^{-1/2}(\Gamma)} = \sqrt{\langle \mathcal{I}_1 (\rho \Lambda g), \rho \Lambda g \rangle_{L^2(\R^2)}} \lesssim \frac{1}{\sqrt{n}}.$$
Of course, $\|g_n\|_{L^2(\Gamma)} \lesssim 1/n$.
Now suppose that $f \in L^2_c(\Gamma)$ and let $d = \int_{\Gamma} f \, d\sigma$. Then $f - dg_n \in L^2_{c,0}(\Gamma)$ and $f - dg_n \to f$ in $\dot{H}^{-1/2}(\Gamma)$, $\E(\Gamma)$, and $L^2(\Gamma)$ as $n \to \infty$. This proves that $L^2_{c,0}(\Gamma)$ is dense in these three spaces, since $L^2_c(\Gamma)$ is.

Finally, suppose again that $f \in L^2_{c,0}(\Gamma)$. If $K$ is any bounded subgraph of $\Gamma$ containing $\supp f$, then $\Si f \in L^2(K)$ by the usual mapping properties of $\Si$ for connected bounded Lipschitz surfaces \cite{Ver84}. Hence we only need to check the behavior of $\Si f$ at infinity to finish the proof. Letting $\Si(r, r')$ be the kernel of $\Si$, note for $r' \in K$ that
$$\Si(r,r') = \frac{1}{4\pi}\frac{1}{|r|(1+ O(|r|^{-1}))} = \frac{1}{4\pi}|r|^{-1}(1+O(|r|^{-1})), \quad r \to \infty.$$
Therefore, since $\int_{\Gamma} f \, d\sigma = 0$, 
$$\Si f(r) = \int_{\Gamma} \Si(r, r') f(r') \, d\sigma(r') = O(|r|^{-2}), \quad r \to \infty.$$
It follows that $\Si f \in L^2(\Gamma)$.
\end{proof}
We are ready to state and prove the main theorem of this section. For the proof, note that the $J$-method, the $K$-method, and the complex method are all equivalent for interpolation of Hilbert spaces \cite{CHM15, McCar92}. We hence simply refer to the interpolation space $(\mathcal{H}_0, \mathcal{H}_1)_{\theta}$ of exponent $0 < \theta < 1$ between two compatible Hilbert spaces $\mathcal{H}_0$ and $\mathcal{H}_1$.
\begin{thm} \label{thm:homsobolev}
Suppose that $\Gamma$ is an unbounded Lipschitz graph. Then the energy space $\E(\Gamma)$ coincides with the homogeneous Sobolev space $\dot{H}^{-1/2}(\Gamma)$, $\E(\Gamma) \simeq \dot{H}^{-1/2}(\Gamma)$. More precisely, the inclusion of $L^2_c(\Gamma)$ into $\dot{H}^{-1/2}(\Gamma)$ extends to an isomorphism of $\E(\Gamma)$ onto $\dot{H}^{-1/2}(\Gamma)$.
\end{thm}
\begin{proof}
The starting point is that $\Si \colon L^2(\Gamma) \to \dot{H}^{1}(\Gamma)$ is an isomorphism \cite[Lemma 3.1]{DK87}. Let $\widetilde{\Lambda} \colon L^2(\Gamma) \to L^2(\R^2)$ denote the unitary given by 
$$\widetilde{\Lambda} f =  \rho^{1/2} \Lambda f, \quad f \in L^2(\Gamma),$$
where $\rho$, as before, is given by $\rho(x,y)= (1 + |\nabla \varphi(x,y)|^2)^{1/2}$. Then
\begin{equation} \label{eq:interpsetup}
M := \mathcal{F} \rho^{-1/2} \widetilde{\Lambda} \Si \widetilde{\Lambda}^{-1} \rho^{-1/2}\mathcal{F}^{-1} \colon L^2(\R^2, d\xi) \to L^2(\R^2, |\xi|^2 \, d\xi).
\end{equation}
is an isomorphism, since multiplication by $\rho^{-1/2}$ on $L^2(\R^2)$ and $\Si \colon L^2(\Gamma) \to \dot{H}^{1}(\Gamma)$ are both isomorphisms. By \eqref{eq:energyproduct}, $M$ is symmetric with respect to the $L^2(\R^2, d\xi)$-pairing. Therefore, by duality, we can reformulate \eqref{eq:interpsetup} by saying that $M$ continuously extends to an isomorphism
\begin{equation} \label{eq:interpsetup2}
M \colon L^2(\R^2, |\xi|^{-2} \, d\xi) \to L^2(\R^2, d\xi).
\end{equation}
By Lemma~\ref{lem:Sdensedef}, $M$ is initially densely defined on 
$$\dom(M) = \mathcal{F} \rho^{1/2} \widetilde{\Lambda} L^2_{c,0}(\Gamma) \subset L^2(\R^2, d\xi) \cap L^2(\R^2, |\xi|^{-2} \,d\xi),$$
and the meaning of \eqref{eq:interpsetup2} is that $M$ extends continuously to an isomorphism. Interpolation between \eqref{eq:interpsetup} and \eqref{eq:interpsetup2} also gives that
\begin{equation} \label{eq:interpsetup3}
M \colon L^2(\R^2, |\xi|^{-1} \, d\xi) \to L^2(\R^2, |\xi| \, d\xi)
\end{equation}
is bounded. It is not, however, possible at this stage to conclude that this operator is an isomorphism. As a consequence of \eqref{eq:energyproduct} and \eqref{eq:interpsetup3} we conclude that
\begin{equation} \label{eq:interpsetup4}
0 < \langle Mf, f \rangle_{L^2(\R^2, \, d\xi)} \lesssim \|f\|_{L^2(\R^2, |\xi|^{-1} \, d\xi)}^2, \quad f \in \dom(M).
\end{equation}

We also want to consider $M$ as an unbounded operator on $L^2(\R^2, d\xi)$. To avoid confusion we call this operator $R$,
$$R \colon L^2(\R^2, d\xi) \to L^2(\R^2, d\xi), \quad Rf = Mf.$$
In view of \eqref{eq:interpsetup2}, we can let the domain of $R$ be
$$\dom(R) = L^2(\R^2, |\xi|^{-2} \, d\xi) \cap L^2(\R^2, d\xi).$$
The positivity of $R$ on $\dom(M)$ extends to $\dom(R)$. To see this, given $f \in \dom(R) \subset L^2(\R^2, |\xi|^{-1} \, d\xi)$, we may by Lemma~\ref{lem:Sdensedef} choose a sequence in $\dom(M)$, approximating $f$ in $L^2(\R^2, |\xi|^{-1} \, d\xi)$. By \eqref{eq:interpsetup3} and \eqref{eq:interpsetup4} we conclude that
$$\langle Rf, f \rangle_{L^2(\R^2, d\xi)} \geq 0, \quad f \in \dom(R).$$
The same argument shows that $R$ is a symmetric operator,
$$\langle Rf, g \rangle_{L^2(\R^2, d\xi)} = \langle f, Rg \rangle_{L^2(\R^2, d\xi)}, \quad f,g \in \dom(R).$$
Since the operator of \eqref{eq:interpsetup} is an isomorphism, the domain of $R^\ast$ is given by
$$\dom(R^\ast) = \{f \in L^2(\R^2, d\xi) \, : \, |\langle f, Rg \rangle_{L^2(\R^2, d\xi)}| \lesssim \|Rg\|_{L^2(\R^2, |\xi|^{2} \, d\xi)}, g \in \dom(R) \}.$$
The range of $R$ being dense in $L^2(\R^2, |\xi|^{2} \, d\xi)$, it follows that
$$\dom(R^\ast) = \dom(R).$$
We conclude that $R$ is a positive self-adjoint operator.

Consider now the Hilbert space $\mathcal{H}_1 = L^2(\R^2, d\xi)$ with its usual norm and $\mathcal{H}_0 = L^2(\R^2, |\xi|^{-2} d\xi)$ with the alternative norm
\begin{equation} \label{eq:altnormM}
\|f\|_{\mathcal{H}_0} = \|M f\|_{\mathcal{H}_1}, \quad f \in \mathcal{H}_0.
\end{equation}
We apply the characterization of the interpolation spaces $(\mathcal{H}_0, \mathcal{H}_1)_\theta$, $0 < \theta < 1$, given by \cite[Theorem 3.3]{CHM15}. It extends the usual characterization given in \cite[Theorem~15.1]{LM72} to the present situation in which $\mathcal{H}_0$ and $\mathcal{H}_1$ are incomparable. The conclusion\footnote{There is a slight mistake in the statement of \cite[Theorem 3.3]{CHM15} concerning $\dom(T)$, but it is easily corrected by inspecting its proof.} is that the relationship
\begin{equation} \label{eq:altnormT}
\langle T^{1/2}f, T^{1/2}g \rangle_{\mathcal{H}_1} = \langle f, g \rangle_{\mathcal{H}_0}, \quad f, g \in \mathcal{H}_0 \cap \mathcal{H}_1
\end{equation}
defines an unbounded, self-adjoint, positive operator $T \colon \mathcal{H}_1 \to \mathcal{H}_1$ whose square root has domain 
$$\dom(T^{1/2}) = \mathcal{H}_0 \cap \mathcal{H}_1 = \dom(R).$$
Furthermore, the norm of the interpolation space $(\mathcal{H}_0, \mathcal{H}_1)_{1/2}$ is given by
$$\|f\|_{(\mathcal{H}_0, \mathcal{H}_1)_{1/2}}^2 = \|T^{1/4}f\|_{\mathcal{H}_1}^2 = \langle T^{1/2}f, f \rangle_{\mathcal{H}_1}, \quad f \in \mathcal{H}_0 \cap \mathcal{H}_1.$$
By \eqref{eq:altnormM} and \eqref{eq:altnormT} we have that
$$\langle T^{1/2}f, T^{1/2}g \rangle_{\mathcal{H}_1} = \langle Rf, Rg \rangle_{\mathcal{H}_1}, \quad f, g \in \dom(T^{1/2}) = \dom(R).$$ 
Since $R$ is also positive and self-adjoint it must be that $R = T^{1/2}$, see for example \cite[Proposition~10.4]{Schm12}.

On the other hand, if we equip $\mathcal{H}_0 = L^2(\R^2, |\xi|^{-2} \, d\xi)$ with the usual norm, we know that the interpolation space is $L^2(\R^2, |\xi|^{-1} \, d\xi)$, and thus
$$\langle Mf, f\rangle_{\mathcal{H}_1} = \langle T^{1/2}f, f \rangle_{\mathcal{H}_1} \simeq \|f\|_{L^2(\R^2, |\xi|^{-1} \, d\xi)}^2, \quad f \in \dom(M).$$
Unraveling the definitions, this means that
$$\|f\|^2_{\E(\Gamma)} = \langle \Si f, f \rangle_{L^2(\Gamma)} \simeq \|\rho \Lambda f\|_{\dot{H}^{-1/2}(\R^2)}^2 = \|f\|_{\dot{H}^{-1/2}(\Gamma)}^2, \quad f\in L^2_{c,0}(\Gamma),$$
where the last equality is given by \eqref{eq:normgamma}.
Since $L^2_{c,0}(\Gamma)$ is dense in $\E(\Gamma)$ and $H^{-1/2}(\Gamma)$ by Lemma~\ref{lem:Sdensedef}, this proves the statement.
\end{proof}
\begin{remark}
When $\Gamma$ is a connected bounded Lipschitz surface, $\Si$ is an isomorphism of $L^2(\Gamma)$ onto the inhomogeneous Sobolev space $H^{1}(\Gamma)$ \cite[Theorem 3.3]{Ver84}, and $\E(\Gamma) \simeq H^{-1/2}(\Gamma)$ in this case.
\end{remark}
\subsection{Single layer potentials and the Dirichlet problem}\label{subsec:singpot}
It is implicit in the proof of Theorem~\ref{thm:homsobolev} that the isomorphism property of $\Si \colon L^2(\Gamma) \to \dot{H}^1(\Gamma)$ extends to the scale of homogeneous Sobolev spaces. When $\Gamma$ is a bounded Lipschitz surface the corresponding result is well known, see for example \cite[Theorem~8.1]{FMM98}.
\begin{cor} \label{cor:Sbdaryisom}
	Let $\Gamma$ be an unbounded Lipschitz graph. For every $0 \leq s \leq 1$,
	$$\Si \colon \dot{H}^{-s}(\Gamma) \to \dot{H}^{1-s}(\Gamma)$$
	is an isomorphism.
\end{cor}
\begin{proof}
	Following the proof of Theorem~\ref{thm:homsobolev}, we see for every $0 < s < 1$ that
	$$\| R^s f \|_{L^2(\R^2, d\xi)} = \|f\|_{(\mathcal{H}_0, \mathcal{H}_1)_{1-s}} \simeq \|f\|_{L^2(\R^2, |\xi|^{-2s} \, d\xi)}, \quad f \in \dom(R).$$
	$R \colon L^2(\R^2, d\xi) \to L^2(\R^2, d\xi)$ has dense range by the isomorphism property of \eqref{eq:interpsetup2}. It follows that $R^s$ extends to an isomorphism $R^s \colon L^{2}(\R^2,|\xi|^{-2s} \, d\xi) \to L^2(\R^2, d\xi)$, and, by duality, to an isomorphism $R^s \colon L^2(\R^2, d\xi) \to L^{2}(\R^2,|\xi|^{2s}).$ Thus $R$ extends to an isomorphism
	$$R = R^{1-s}R^s \colon L^{2}(\R^2,|\xi|^{-2s} \, d\xi) \to L^{2}(\R^2,|\xi|^{2(1-s)} \, d\xi).$$
	This is equivalent to the statement of the corollary.
\end{proof}
Consider the homogeneous Sobolev spaces on $\Gamma_+$ and $\Gamma_-$,
$$\dot{H}^1(\Gamma_\pm) = \left\{F \in L^2_{\textrm{loc}}(\Gamma_{\pm}) \, : \, \|F\|_{\dot{H}^1(\Gamma_\pm)}^2 = \int_{\Gamma_\pm} |\nabla F|^2 \, dV < \infty \right\}.$$
These are Hilbert spaces as quotient spaces over the constant functions. The subspaces of harmonic functions are given by
$$\dot{H}^1_h(\Gamma_\pm) = \{U \in \dot{H}^1(\Gamma_\pm) \, : \, \Delta U = 0 \textrm{ in } \Gamma_{\pm} \}.$$
It follows from equations \eqref{eq:energyproduct}-\eqref{eq:energyinner} that, evaluating $\Si f$ in either $\Gamma_+$ or $\Gamma_-$ for a charge $f$, $\Si$ extends to bounded maps
$$\Si \colon \E(\Gamma) \to \dot{H}^1_h(\Gamma_\pm).$$
By the trace inequality \cite[Theorem~2.4]{EL12} and the method of \cite{Din96}, there are (unique) continuous traces $\Tr_{\pm} \colon \dot{H}^1(\Gamma_\pm) \to \dot{H}^{1/2}(\Gamma)$. By the corresponding result for bounded Lipschitz surfaces $\Gamma$ \cite{Cost07}, and by considering smooth cut-off functions, we see that $\Tr_{\pm} \Si f = \Si f$ for $f \in L^2_{c}(\Gamma)$. By Corollary~\ref{cor:Sbdaryisom}, both sides of this equation extend continuously to $\dot{H}^{-1/2}(\Gamma) \simeq \E(\Gamma)$, and we conclude that
$$\Tr_{+} \Si f = \Tr_{-} \Si f = \Si f \in \dot{H}^{1/2}(\Gamma), \quad f \in \E(\Gamma).$$
This leads to the following result on the interior Dirichlet problem. Of course, we could equally well make the analogous statement for the exterior Dirichlet problem.
\begin{cor} \label{cor:dirichletE}
The trace $\Tr_{+} \colon \dot{H}^1_h(\Gamma_+) \to \dot{H}^{1/2}(\Gamma)$ is an isomorphism. That is, the Dirichlet problem
$$\begin{cases} \int_{\Gamma_{+}} |\nabla U|^2 \, dV < \infty, \\
\Delta U = 0 \textrm{ in } \Gamma_{+}, \\
\Tr_{+} U = g \in \dot{H}^{1/2}(\Gamma),\end{cases}$$
is well-posed. The unique solution $U$ is given by a single layer potential, $U = \Si f$, where $f \in \E(\Gamma)$. Hence $\Si \colon \E(\Gamma) \to \dot{H}^1_h(\Gamma_{+})$ is an isomorphism.
\end{cor}
\begin{proof}
	Given $g \in \dot{H}^{1/2}(\Gamma)$, there is by Corollary~\ref{cor:Sbdaryisom} an $f \in \E(\Gamma)$ such that $\Si f = g$ in $\dot{H}^{1/2}(\Gamma)$. Then $U = \Si f$ in $\dot{H}_h^1(\Gamma_+)$ solves the Dirichlet problem. Uniqueness is given by \cite[Theorem~7.1.2]{Medk18}. Corollary~\ref{cor:Sbdaryisom} shows that $\Si \colon \E(\Gamma) \to \dot{H}^1_h(\Gamma_{+})$ is injective, while surjectivity is given by well-posedness of the Dirichlet problem.
\end{proof}
\section{The energy space spectrum} \label{sec:Espec}
We now return to the situation where $\Gamma = \Gamma_\alpha$ is the boundary of a wedge of opening angle $\alpha$. Recall from Section~\ref{sec:conv} that
$$\K_\alpha = \begin{pmatrix} 
	0 & K_\alpha \\
	K_\alpha & 0 
\end{pmatrix}, \quad \Si_\alpha = \begin{pmatrix} 
S_0 & S_\alpha \\
S_\alpha & S_0 
\end{pmatrix},$$
where $K_\alpha f = f \star k_\alpha$ and, by equation~\eqref{eq:salphaformula}, $S_\beta f = V_1f \star s_\beta$, with convolution kernels given by
$$k_\alpha(s,t) = -\frac{\sin \alpha}{2\pi}\frac{1}{(1+s^2 - 2s\cos \alpha + t^2)^{3/2}}$$
and
$$s_\beta(s,t) = \frac{1}{4\pi} \frac{1}{(1+s^2 - 2s \cos \beta + t^2)^{1/2}}.$$ 
Here, as before, $V_\gamma$, $\gamma \in \R$, is the operator of multiplication by $\Delta^{-\gamma} = x^\gamma$. 

For technical purposes, we begin by establishing some mapping properties of $\Si_\alpha$ and $\K_\alpha$, refraining from working out the much more general statement that could be given. For $1 \leq p < \infty$ we write $L^{p,a}(dx \, dz) = L^p(x^a \, dx \, dz)$ and $L^{p,a}(\Gamma_\alpha) = L^{p,a}(dx \, dz) \oplus L^{p,a}(dx \, dz)$.
\begin{lem} \label{lem:mapping}
	The following operators are bounded:
	\begin{align*}
	&\Si_\alpha \colon L^{2, 3/4}(\Gamma_\alpha) \to L^{4, -1/2}(\Gamma_\alpha),  &\K_\alpha \colon L^{2, 3/4}(\Gamma_\alpha) \to L^{2, 3/4}(\Gamma_\alpha), \\ &\K_\alpha \colon L^{4/3, 1/8}(\Gamma_\alpha) \to L^{4/3, 1/8}(\Gamma_\alpha), &\K_\alpha^\ast \colon L^{4, -1/2}(\Gamma_\alpha) \to L^{4, -1/2}(\Gamma_\alpha).
	\end{align*}
\end{lem}
\begin{proof}
Note first that $V_a s_\beta \in L^q(G)$ for $\beta \in \{0,\alpha\}$ if and only if $1 < q < 2$ and $0 < aq < q-1$. To see this, note that $0 \leq s_\alpha \leq s_0$ and that if $q > 1$, then
	$$\|V_as_0\|_{L^q}^q = \frac{1}{4\pi} \int_0^\infty \int_{-\infty}^\infty \frac{s^{aq}}{((1-s)^q + t^2)^{q/2}} \, dt \, \frac{ds}{s} = c_q \int_{0}^\infty \frac{s^{aq-1}}{|1-s|^{q-1}} \, ds,$$
	where $c_q > 0$ is a constant. 
	We now let $q = 4/3$. By Young's inequality, Lemma~\ref{lem:young},
	\begin{align*}
	\|S_\beta f\|_{L^{4,4a-1}(dx \, dz)} &= \|V_{1+a}f \star V_a s_\beta \|_{L^4(G)} \leq \|V_{3/4+a} f\|_{L^2(G)} \|V_a s_\beta\|_{L^{4/3}(G)} \\ &\lesssim \|f\|_{L^{2, 1/2 + 2a}(dx \, dz)},
	\end{align*}
	whenever $\beta \in \{0,\alpha\}$ and $0 < a < 1/4$. This yields that $\Si_\alpha \colon L^{2, 3/4}(\Gamma_\alpha) \to L^{4, -1/2}(\Gamma_\alpha)$ is bounded upon choosing $a = 1/8$.
	
	That $\K_\alpha \colon L^{2, 3/4}(\Gamma_\alpha) \to L^{2, 3/4}(\Gamma_\alpha)$ is bounded is part of Lemma~\ref{lem:l2norm}. The boundedness of $\K_\alpha \colon L^{4/3, 1/8}(\Gamma_\alpha) \to L^{4/3, 1/8}(\Gamma_\alpha)$ also follows from (the proof of) that lemma, since it shows that $V_{27/32}k_\alpha \in L^1(G)$, and by Young's inequality
	\begin{align*}
	\|K_\alpha f\|_{L^{4/3, 1/8}(dx \, dz)} &= \|V_{27/32}f \star V_{27/32}k_\alpha \|_{L^{4/3}(G)} \\ &\leq \|V_{27/32}f\|_{L^{4/3}(G)}\|V_{27/32}k_\alpha\|_{L^1(G)} \lesssim \|f\|_{L^{4/3,1/8}(dx \, dz)}.
	\end{align*}
	 This last estimate also proves that $\K_\alpha^\ast \colon L^{4, -1/2}(\Gamma_\alpha) \to L^{4, -1/2}(\Gamma_\alpha)$ is bounded, since $L^{4, -1/2}(\Gamma_\alpha)$ is the dual space of $L^{4/3, 1/8}(\Gamma_\alpha)$ under the $L^2(\Gamma)$-pairing.
\end{proof}
The lemma allows us to motivate the Plemelj formula $\Si \K = \K^\ast \Si$ for the unbounded Lipschitz graph $\Gamma_\alpha$.
\begin{lem} \label{lem:plemelj}
The Plemelj formula is valid for $\Gamma_\alpha$ when either side of the equation is interpreted as a bounded operator from $L^{2, 3/4}(\Gamma_\alpha)$ into $L^{4,-1/2}(\Gamma_\alpha)$. That is,
$$\Si_\alpha \K_\alpha = \K_\alpha^\ast \Si_\alpha \colon L^{2, 3/4}(\Gamma_\alpha) \to L^{4,-1/2}(\Gamma_\alpha).$$
\end{lem}
\begin{proof}
Choose a sequence $(\Gamma^j)_{j=1}^\infty$ of bounded connected Lipschitz surfaces such that $(\Gamma^j \cap \Gamma_\alpha)_{j=1}^\infty$ is an increasing exhaustion of $\Gamma_\alpha$. The choice of sequence can be made so that for any compact set $K \subset \Gamma_\alpha$ it holds for sufficiently large $j$ that
$$\dist(K, \Gamma^j \setminus \Gamma_\alpha) \simeq j, \quad \int_{\Gamma^j} \, d\sigma_j \simeq j^2,$$
where $d\sigma_j$ denotes the surface measure of $\Gamma^j$.
Suppose that $f,g \in L^2_c(\Gamma_\alpha)$. For sufficiently large $j$ we can understand $f$ and $g$ as functions on $\Gamma^j$, and then, by the Plemelj formula for bounded domains \cite[Theorem~3.3]{Ver84}, 
\begin{equation} \label{eq:plemeljpf1}
\langle \K_{\Gamma^j} f, \Si_{\Gamma^j} g \rangle_{L^2(\Gamma^j)} = \langle \Si_{\Gamma^j} f, \K_{\Gamma^j} g \rangle_{L^2(\Gamma^j)},
\end{equation}
where $\K_{\Gamma^j}$ and $\Si_{\Gamma^j}$ denote the layer potentials of $\Gamma^j$.
Note that 
\begin{equation} \label{eq:plemeljpf2}
\langle \K_{\Gamma^j} f, \Si_{\Gamma^j} g \rangle_{L^2(\Gamma^j)} = \langle \K_{\alpha} f, \Si_{\alpha} g \rangle_{L^2(\Gamma^j \cap \Gamma_\alpha)} + \langle \K_{\Gamma^j} f, \Si_{\Gamma^j} g \rangle_{L^2(\Gamma^j \setminus  \Gamma_\alpha)},
\end{equation}
where
$$|\langle \K_{\Gamma^j} f, \Si_{\Gamma^j} g \rangle_{L^2(\Gamma^j \setminus  \Gamma_\alpha)}| \lesssim \frac{1}{j^3}\int_{\Gamma^j} \, d\sigma_j \to 0, \quad j \to \infty.$$
Since, by Lemma~\ref{lem:mapping}, $\K_{\alpha} f \in L^{4/3, 1/8}(\Gamma_\alpha)$ and $\Si_{\alpha} g \in L^{4, -1/2}(\Gamma_\alpha)$, we deduce from \eqref{eq:plemeljpf2} that
$$\langle \K_{\Gamma^j} f, \Si_{\Gamma^j} g \rangle_{L^2(\Gamma^j)} \to \langle \K_{\alpha} f, \Si_{\alpha} g \rangle_{L^2(\Gamma_\alpha)}, \quad j \to \infty.$$ Similarly, $\langle \Si_{\Gamma^j} f, \K_{\Gamma^j} g \rangle_{L^2(\Gamma^j)} \to \langle \Si_{\alpha} f, \K_{\alpha} g \rangle_{L^2(\Gamma_\alpha)}$. By \eqref{eq:plemeljpf1} we conclude that
\begin{equation} \label{eq:plemeljpf3}
\langle \K_{\alpha} f, \Si_{\alpha} g \rangle_{L^2(\Gamma_\alpha)} = \langle \Si_{\alpha} f, \K_{\alpha} g \rangle_{L^2(\Gamma_\alpha)}, \quad f,g \in L^2_c(\Gamma).
\end{equation}
By Lemma~\ref{lem:mapping}, the operators $\Si_{\alpha} \K_{\alpha}, \K_{\alpha}^\ast \Si_{\alpha} \colon L^{2, 3/4}(\Gamma_\alpha) \to L^{4,-1/2}(\Gamma_\alpha)$ are bounded. Hence we infer from \eqref{eq:plemeljpf3} that they are equal, $\Si_{\alpha} \K_{\alpha} = \K_{\alpha}^\ast \Si_{\alpha}.$
\end{proof}

Lemmas~\ref{lem:mapping} and \ref{lem:plemelj} let us define $\K_\alpha$ as an unbounded symmetric operator on $\E(\Gamma_\alpha)$. We will later see that $\K_{\alpha} \colon \E(\Gamma_\alpha) \to \E(\Gamma_\alpha)$ is bounded (and hence self-adjoint).
\begin{lem} \label{lem:KEdef}
Let 
$$\mathcal{D}(\Gamma_\alpha) = L^{2, 3/4}(\Gamma_\alpha) \cap L^{4/3, 1/8}(\Gamma_\alpha).$$ 
Then $L^2_{c}(\Gamma_\alpha) \subset \mathcal{D}(\Gamma_\alpha) \subset \E(\Gamma_\alpha)$, the second inclusion understood in the natural way such that
\begin{equation} \label{eq:eformulaholds}
\langle f, g \rangle_{\E(\Gamma_\alpha)} = \langle \Si_{\alpha} f, g \rangle_{L^2(\Gamma_\alpha)}, \quad f,g \in \mathcal{D}(\Gamma_\alpha).
\end{equation}
 Furthermore, $\K_{\alpha} \mathcal{D}(\Gamma_\alpha) \subset \mathcal{D}(\Gamma_\alpha)$. Thus $\K_{\alpha} \colon \E(\Gamma_\alpha) \to \E(\Gamma_\alpha)$ is densely defined with domain $\mathcal{D}(\Gamma_\alpha)$, and this operator is symmetric.
\end{lem} 
\begin{proof}
	By H\"older's inequality, $L^2_{c}(\Gamma_\alpha) \subset \mathcal{D}(\Gamma_\alpha)$. Let $(\Gamma^j)_{j=1}^\infty$ be an increasing exhausting sequence of compact subsets of 
	$$\Gamma_\alpha \setminus \{(0,0,z)\, : \, z\in \R\}.$$
	 Given $f \in \mathcal{D}(\Gamma_\alpha)$ and $j \geq 1$, let $f_j = \chi_{\Gamma^j} f \in L^2_c(\Gamma_\alpha)$. Then $f_j \to f$ in $\mathcal{D}(\Gamma_\alpha)$. Hence $(f_j)_{j=1}^\infty$ is a Cauchy sequence in $\mathcal{E}(\Gamma_\alpha)$, since, by Lemma~\ref{lem:mapping} and the duality between $L^{4, -1/2}(\Gamma_\alpha)$ and $L^{4/3, 1/8}(\Gamma_\alpha)$,
	\begin{equation} \label{eq:KEdef}
	\|f_j - f_k\|_{\mathcal{E}(\Gamma_\alpha)}^2 \leq \|\Si_{\alpha}(f_j - f_k)\|_{L^{4, -1/2}(\Gamma_\alpha)}\|f_j - f_k\|_{L^{4/3, 1/8}(\Gamma_\alpha)} \to 0, \quad j,k \to \infty.
	\end{equation}
	Therefore $(f_j)_{j=1}^\infty$ represents an element of $\E(\Gamma_\alpha)$ such that
	$$\lim_{j \to \infty} \langle f_j , g \rangle_{L^2(\Gamma_\alpha)} = \langle f, g \rangle_{L^2(\Gamma_\alpha)}, \quad g \in \Lambda^{-1}C_c^\infty(\R^2).$$
	Since $\E(\Gamma_\alpha) \simeq \dot{H}^{-1/2}(\Gamma_\alpha)$ by Theorem~\ref{thm:homsobolev}, this shows that $(f_j)_{j=1}^\infty$ corresponds to the element $f \in \dot{H}^{-1/2}(\Gamma_\alpha)$. In particular, the map 
	$$\mathcal{D}(\Gamma_\alpha) \ni f \mapsto (f_j)_{j=1}^\infty \in \E(\Gamma_\alpha)$$ is injective. Hence it is justified to consider $\mathcal{D}(\Gamma_\alpha)$ a linear subspace of $\E(\Gamma_\alpha)$. Equation \eqref{eq:eformulaholds} follows from \eqref{eq:KEdef} and Lemma~\ref{lem:mapping}.
	
	That $\K_{\alpha} \mathcal{D}(\Gamma_\alpha) \subset \mathcal{D}(\Gamma_\alpha)$ is a consequence of Lemma~\ref{lem:mapping}, and the symmetry of $\K_{\alpha} \colon \E(\Gamma_\alpha) \to \E(\Gamma_\alpha)$ is Lemma~\ref{lem:plemelj}.
\end{proof}
Our next goal is to prove that $\K_{\alpha} \colon \E(\Gamma_\alpha) \to \E(\Gamma_\alpha)$ is actually bounded and to give the correct estimate for its norm. For $\gamma \in \R$, let
$$\mathbb{V}_\gamma  = \begin{pmatrix} 
V_\gamma & 0 \\
0 & V_\gamma 
\end{pmatrix},$$
and consider for $0 < a \leq 1$ the space $\mathcal{E}_a(\Gamma_\alpha) = \mathbb{V}_a \E(\Gamma_\alpha)$, the completion of $\mathbb{V}_a \mathcal{D}(\Gamma_\alpha)$ under the scalar product
\begin{equation} \label{eq:Eascalar}
\langle f, g \rangle_{\E_a(\Gamma_\alpha)} := \langle \mathbb{V}_{-a}  \Si_{\alpha} \mathbb{V}_{-a} f, g \rangle_{L^{2}(\Gamma_\alpha)} = \langle \mathbb{V}_{1-a} \Si_{\alpha} \mathbb{V}_{-a} f, g \rangle_{L^{2,-1}(\Gamma_\alpha)}.
\end{equation}
The following is obvious by definition and Lemma~\ref{lem:KEdef}.
\begin{lem} \label{lem:Eaunitary}
For $0 < a \leq 1$, $\K_{\alpha} \colon \E(\Gamma_\alpha) \to \E(\Gamma_\alpha)$ is unitarily equivalent to $\mathbb{V}_a \K_{\alpha} \mathbb{V}_{-a} \colon \E_a(\Gamma_\alpha) \to \E_a(\Gamma_\alpha)$, the latter operator having domain $\mathbb{V}_a \mathcal{D}(\Gamma_\alpha)$.
\end{lem}
We now run a symmetrization argument (cf. \cite[Theorem~2.2]{HSS05}) to prove that 
$$\|\K_\alpha\|_{B(\E(\Gamma_\alpha))} \leq \|\mathbb{V}_a \K_{\alpha} \mathbb{V}_{-a}\|_{B(L^{2,-1}(\Gamma_\alpha))} = \|\K_{\alpha}\|_{B(L^{2,2a-1}(\Gamma_\alpha))}, \quad 0 < a < 1.$$
The norm on right-hand side was computed in Theorem~\ref{thm:main1}, and taking $a \to 1$ yields the following result.
\begin{thm} \label{thm:Enormbound}
$\K_{\alpha} \colon \E(\Gamma_\alpha) \to \E(\Gamma_\alpha)$ is bounded with norm
\begin{equation} \label{eq:Enormbound}
\|\K_\alpha\|_{B(\E(\Gamma_\alpha))} \leq |1 - \alpha/\pi|.
\end{equation}
\end{thm}
\begin{proof}
	Let $0 < a < 1$, and consider a function $f \in L^2(\Gamma_\alpha)$, compactly supported in $\Gamma_\alpha \setminus \{(0,0,z) \, : \, z \in \R\}$ and satisfying
	$$\int_{\Gamma_\alpha} \mathbb{V}_{-a} f \, d\sigma = 0.$$
	  The space of such functions is included in $\mathbb{V}_a \mathcal{D}(\Gamma_\alpha)$ and dense in $\mathcal{E}_a(\Gamma_\alpha)$, which follows from the fact that $\mathbb{V}_a L^2_{c}(\Gamma)$ is dense in $\E_a(\Gamma_\alpha)$, together with a small modification of the proof of Lemma~\ref{lem:Sdensedef}. Let $K \subset \Gamma_\alpha \setminus \{(0,0,z) \, : \, z \in \R\}$ be a compact set such that $f$ is compactly supported in the interior of $K$.  Then 
	  $$\chi_K \mathbb{V}_{1-a} \Si_\alpha \mathbb{V}_{-a} f \in L^{2,-1}(\Gamma_\alpha)$$
	   by the usual mapping properties of the single layer potential on a bounded connected Lipschitz surface. For $r \in \Gamma_\alpha \setminus K$, we have by Lemma~\ref{lem:Sdensedef} that $\Si_\alpha \mathbb{V}_{-a} f(r) = O((1+|r|^{2})^{-1})$. Hence
	  \begin{align*}
	  \|\chi_{\Gamma_\alpha \setminus K} \mathbb{V}_{1-a} \Si_\alpha \mathbb{V}_{-a} f\|_{L^{2,-1}(\Gamma_\alpha)} &\lesssim \int_0^\infty \int_{-\infty}^\infty \frac{x^{2(1-a)}}{(1+x^2+z^2)^2} \, dz\, \frac{dx}{x} \\
	  &= \frac{\pi}{2}\int_0^\infty \frac{x^{2(1-a)}}{(1+x^2)^{3/2}} \, \frac{dx}{x} < \infty,
	  \end{align*}
  	  since $0 < a < 1$. We conclude that 
  	  $$\mathbb{V}_{1-a} \Si_\alpha \mathbb{V}_{-a}f \in L^{2,-1}(\Gamma_\alpha).$$
  	  
  	  Since $f \in \mathbb{V}_a \mathcal{D}(\Gamma_\alpha)$ and the symmetric (at this stage possibly unbounded) operator $\mathbb{V}_a \K_{\alpha} \mathbb{V}_{-a} \colon \E_a(\Gamma_\alpha) \to \E_a(\Gamma_\alpha)$ preserves $\mathbb{V}_a \mathcal{D}(\Gamma_\alpha)$, we find that
  	  $$\|\mathbb{V}_a \K_{\alpha} \mathbb{V}_{-a} f\|_{\E_a(\Gamma_\alpha)} = \langle f, \mathbb{V}_a \K_{\alpha}^2 \mathbb{V}_{-a}f\rangle_{\E_a(\Gamma_\alpha)}^{1/2} \leq \|f\|_{\E_a(\Gamma_\alpha)}^{1/2} \|\mathbb{V}_a \K_{\alpha}^2 \mathbb{V}_{-a} f\|_{\E_a(\Gamma_\alpha)}^{1/2}.$$
  	  Repeating the estimate inductively gives us that
  	  $$\|\mathbb{V}_a \K_{\alpha} \mathbb{V}_{-a} f\|_{\E_a(\Gamma_\alpha)} \leq \|f\|_{\E_a(\Gamma_\alpha)}^{1 - 2^{-j}} \|\mathbb{V}_a \K_{\alpha}^{2^j} \mathbb{V}_{-a} f \|_{\E_a(\Gamma_\alpha)}^{2^{-j}}, \quad j \geq 1.$$
  	   Estimating the right-most norm with the help of \eqref{eq:Eascalar} yields that
  	  \begin{multline*}
  	  \|\mathbb{V}_a \K_{\alpha} \mathbb{V}_{-a} f\|_{\E_a(\Gamma_\alpha)} \leq \|f\|_{\E_a(\Gamma_\alpha)}^{1 - 2^{-j}}  \|\mathbb{V}_{1-a} \Si_{\alpha} \mathbb{V}_{-a} f\|_{L^{2,-1}(\Gamma_\alpha)}^{2^{-j-1}} \|\mathbb{V}_a \K_{\alpha}^{2^{j+1}} \mathbb{V}_{-a} f\|_{L^{2,-1}(\Gamma_\alpha)}^{2^{-j-1}} \\
  	  \leq \|f\|_{\E_a(\Gamma_\alpha)}^{1 - 2^{-j}}  \|\mathbb{V}_{1-a} \Si_{\alpha} \mathbb{V}_{-a} f\|_{L^{2,-1}(\Gamma_\alpha)}^{2^{-j-1}} \|f\|_{L^{2,-1}(\Gamma_\alpha)}^{2^{-j-1}} \|\mathbb{V}_a \K_{\alpha} \mathbb{V}_{-a}\|_{B(L^{2,-1}(\Gamma_\alpha))}.
  	  \end{multline*}
  	  Since $f$ and $\mathbb{V}_{1-a} \Si_{\alpha} \mathbb{V}_{-a} f$ belong to $L^{2,-1}(\Gamma_\alpha)$, in letting $j \to \infty$ we conclude that
  	  $$\|\K_\alpha\|_{B(\E(\Gamma_\alpha))} = \|\mathbb{V}_a \K_{\alpha} \mathbb{V}_{-a}\|_{B(\E_a(\Gamma_\alpha))}\leq \|\mathbb{V}_a \K_{\alpha} \mathbb{V}_{-a}\|_{B(L^{2,-1}(\Gamma_\alpha))},$$
  	  as promised. In particular, $\K_\alpha \colon \E(\Gamma_\alpha) \to \E(\Gamma_\alpha)$ is bounded. By Theorem~\ref{thm:main1},
  	  $$\|\mathbb{V}_a \K_{\alpha} \mathbb{V}_{-a}\|_{B(L^{2,-1}(\Gamma_\alpha))} = \left|\frac{\sin\left( (1-a)(\pi - \alpha) \right)}{\sin\left( (1-a) \pi \right)}\right|.$$
  	  We obtain \eqref{eq:Enormbound} when we let $a \to 1$.
\end{proof}  
\begin{remark}
The reason for not directly considering $a=1$ in the proof is that it appears difficult to find an appropriate dense class of functions $f$ for which $\mathbb{V}_{1-a} \Si_\alpha \mathbb{V}_{-a}f = \Si_\alpha \mathbb{V}_{-1} f \in L^{2,-1}(\Gamma_\alpha)$.
\end{remark}
We are finally in a position to determine the spectrum of $\K_\alpha \colon \E(\Gamma_\alpha) \to \E(\Gamma_\alpha)$. Let us begin by describing an unsuccessful approach, which nonetheless is illuminating. By inspection of \eqref{eq:Kformula} we see that
\begin{equation} \label{eq:Kastformula}
\K_{\alpha}^\ast = \mathbb{V}_1 \K_\alpha \mathbb{V}_{-1}.
\end{equation}
 Hence Lemma~\ref{lem:Eaunitary} for $a=1$ says that $\K_\alpha \colon \E(\Gamma_\alpha) \to \E(\Gamma_\alpha)$ is unitarily equivalent to $\K_{\alpha}^\ast \colon \E_1(\Gamma_\alpha) \to \E_1(\Gamma_\alpha)$. The scalar product of $\E_1(\Gamma_\alpha)$ is given by
\begin{equation} \label{eq:E1scalar}
\langle f, g \rangle_{\E_1(\Gamma_\alpha)} = \langle \widetilde{\Si}_\alpha f, g \rangle_{L^{2,-1}(\Gamma_\alpha)}, \quad f, g \in \mathbb{V}_1 \mathcal{D}(\Gamma_\alpha),
\end{equation}
where $\widetilde{\Si}_\alpha = \Si_\alpha \mathbb{V}_{-1}$. Note that $\widetilde{\Si}_\alpha$ is a block matrix of convolution operators on the group $G$. Plemelj's formula, Lemma~\ref{lem:plemelj}, says that $\widetilde{\Si}_\alpha$ and $\K_\alpha^\ast$ commute,
$$\widetilde{\Si}_\alpha \K_\alpha^\ast f = \Si_{\alpha} \K_\alpha \mathbb{V}_{-1}f = \K_\alpha^\ast\widetilde{\Si}_\alpha f, \quad f \in \mathbb{V}_1\mathcal{D}(\Gamma_\alpha).$$
Suppose that we could construct a suitable square root of $\widetilde{\Si}_\alpha$ which commutes with $\K_\alpha^\ast$. Then, in view of \eqref{eq:E1scalar}, it should be possible to conclude that $$(\widetilde{\Si}_\alpha)^{1/2} \colon \E_1(\Gamma_\alpha) \to L^{2,-1}(\Gamma_\alpha)$$
is a unitary map. It would hence follow that $\K_\alpha \colon \E(\Gamma_\alpha) \to \E(\Gamma_\alpha)$ is unitarily equivalent to $\K_\alpha^\ast \colon L^{2,-1}(\Gamma_\alpha) \to L^{2,-1}(\Gamma_\alpha)$, which in turn, by \eqref{eq:Kastformula}, is unitarily equivalent to $\K_\alpha \colon L^{2,1}(\Gamma_\alpha) \to L^{2,1}(\Gamma_\alpha)$. We have already computed the spectrum of this latter operator in Theorem~\ref{thm:main1}.

Unfortunately, while the scalar product \eqref{eq:E1scalar} is a positive definite form, it is not clear to the author how to construct the desired square root. Therefore we will compare $\K_{\alpha}^\ast \colon \mathcal{E}_1(\Gamma_\alpha) \to \mathcal{E}_1(\Gamma_\alpha)$ with $\K_{\alpha}^\ast \colon L^{2,-1}(\Gamma_\alpha) \to L^{2,-1}(\Gamma_\alpha)$ in an indirect way, yielding slightly less information.

\begin{thm} \label{thm:main2}
	The bounded self-adjoint operator $\K_\alpha \colon \E(\Gamma_\alpha) \to \E(\Gamma_\alpha)$ satisfies that
	$$\sigma(\K_\alpha, \E(\Gamma_\alpha)) = \sigma_{\textrm{ess}}(\K_\alpha, \E(\Gamma_\alpha)) = \sigma(\K_\alpha, L^{2,1}(\Gamma_\alpha)) = [-|1 - \alpha/\pi|, |1 - \alpha/\pi|].$$
\end{thm}
\begin{proof}
	Let $I = [-|1 - \alpha/\pi|, |1 - \alpha/\pi|]$. By Theorem~\ref{thm:Enormbound} we already know that 
	\begin{equation*} 
	\sigma(\K_\alpha, \E(\Gamma_\alpha)) =  \sigma(\K_\alpha^\ast, \E_1(\Gamma_\alpha)) \subset I.
	\end{equation*}
	Suppose that $A$ is a non-empty open subset of $I \setminus \sigma(\K_\alpha^\ast, \E_1(\Gamma_\alpha)).$ Since, by Theorem~\ref{thm:main1},
	$$\sigma(\K_\alpha^\ast, L^{2,-1}(\Gamma_\alpha)) = \sigma(\K_\alpha, L^{2,1}(\Gamma_\alpha)) = I$$
	we can, by the functional calculus of the self-adjoint operator $\K_\alpha^\ast \colon L^{2,-1}(\Gamma_\alpha) \to L^{2,-1}(\Gamma_\alpha)$, choose a continuous function $h$ with compact support in $A$ and a function $f \in \mathbb{V}_1 L^2_c(\Gamma_\alpha)$ such that $h(\K_\alpha^\ast|_{L^{2,-1}})f \neq 0$. Let $(h_n)_{n=1}^\infty$ be a sequence of polynomials such that $h_n \to h$ uniformly on $I$ as $n \to \infty$. Then $$h_n(\K_\alpha^\ast|_{L^{2,-1}})f \to h(\K_\alpha^\ast|_{L^{2,-1}})f$$
	 in $L^{2,-1}(\Gamma_\alpha)$, as n $\to \infty$. Accordingly, let $g \in \mathbb{V}_{-1} \Lambda^{-1} C_c^\infty(\R^2) \cap L^{2,1}(\Gamma_\alpha)$ be such that 
	\begin{equation} \label{eq:main2}
	\langle h_n(\K_\alpha^\ast|_{L^{2,-1}})f, g \rangle_{L^2(\Gamma_\alpha)} \to c \neq 0, \quad n \to \infty.
	\end{equation}
	On the other hand, $g \in (\E_1(\Gamma_\alpha))^\ast$, the dual space understood with respect to the $L^2(\Gamma_\alpha)$-pairing, since $\E_1(\Gamma_\alpha) \simeq \mathbb{V}_1 \dot{H}^{-1/2}(\Gamma_\alpha)$ by Theorem~\ref{thm:homsobolev}. Furthermore,
	$$f \in \mathbb{V}_1 \mathcal{D}(\Gamma_\alpha) \cap L^{2,-1}(\Gamma_\alpha) \subset \E_1(\Gamma_\alpha) \cap L^{2,-1}(\Gamma_\alpha),$$
	and $\K_\alpha^\ast$ preserves the space $\mathbb{V}_1 \mathcal{D}(\Gamma_\alpha) \cap L^{2,-1}(\Gamma_\alpha)$, by \eqref{eq:Kastformula} and Lemmas~\ref{lem:l2norm} and \ref{lem:KEdef}. Since $h_n \to 0$ uniformly on $\sigma(\K_\alpha^\ast, \E_1(\Gamma_\alpha))$, we conclude that
	$$\langle h_n(\K_\alpha^\ast|_{L^{2,-1}})f, g \rangle_{L^2(\Gamma_\alpha)} = \langle h_n(\K_\alpha^\ast|_{\E_1})f, g \rangle_{L^2(\Gamma_\alpha)} \to 0, \quad n \to \infty,$$
	a contradiction to \eqref{eq:main2}. Hence it must have been that 
	$$\sigma(\K_\alpha, \E(\Gamma_\alpha)) = I.$$
		Since the spectrum of the self-adjoint operator $\K_\alpha \colon \E(\Gamma_\alpha) \to \E(\Gamma_\alpha)$ is an interval, a set without isolated points, it is of course essential.
\end{proof}
\begin{remark}
The statement of Theorem~\ref{thmx:B} follows by combining Theorems~\ref{thm:main1} and \ref{thm:main2}.
\end{remark}
To apply Theorem~\ref{thm:main2} to the transmission problem, we need to define the normal derivatives $\partial_n^{+} V$ and $\partial_n^{-} V$ of interior and exterior approach for harmonic functions $V \in \dot{H}^1_h(\Gamma_{\alpha,\pm})$, see Section~\ref{subsec:singpot}. Suppose that $U = \Si_\alpha f$ and $V = \Si_{\alpha} g$ for some charges $f,g \in L^2_{c,0}(\Gamma_\alpha)$. Then, by the usual approximation argument with bounded Lipschitz domains, Green's formula 
$$\langle \nabla U, \nabla V \rangle_{L^2(\Gamma_{\alpha, \pm})} = \pm \int_{\Gamma_\alpha} \Tr_{\pm} U \overline{\partial_n^{\pm} V} \, d\sigma = \pm \langle \Tr_{\pm}U, \partial_n^{\pm} V \rangle_{L^2(\Gamma_\alpha)}$$
holds, where the normal derivatives $\partial_n^{\pm} V$ are given by \eqref{eq:nttrace}. Recall that $L^2_{c,0}(\Gamma_{\alpha})$ is dense in $\E(\Gamma_\alpha)$ and that $\Si_{\alpha} \colon \E(\Gamma_\alpha) \to \dot{H}^1_h(\Gamma_{\alpha, \pm})$ and $\Tr_{\pm} \colon \dot{H}^1_h(\Gamma_{\alpha,\pm}) \to \dot{H}^{1/2}(\Gamma_{\alpha})$ are isomorphisms, by Lemma~\ref{lem:Sdensedef} and Corollary~\ref{cor:dirichletE}, respectively. It follows that $\partial_n^{\pm} V \in \E(\Gamma_{\alpha}) \simeq (\dot{H}^{1/2}(\Gamma_{\alpha}))^\ast$, and, furthermore, that $\partial_n^{\pm}$ extends continuously to a bounded map 
$$\partial_n^{\pm} \colon \dot{H}^1_h(\Gamma_{\alpha,\pm}) \to \E(\Gamma_{\alpha}).$$ Since $\K_\alpha \colon \E(\Gamma_{\alpha}) \to \E(\Gamma_{\alpha})$ is bounded, the jump formulas \eqref{eq:L2jumpformulas} extend continuously to $\E(\Gamma_\alpha)$,
\begin{equation} \label{eq:Ejumpformulas}
\partial_n^{+} \Si_\alpha f = \frac{1}{2}(f - \K_\alpha f), \quad \partial_n^{-} \Si_\alpha f(r) = \frac{1}{2}(-f - \K_\alpha f), \quad f \in \E(\Gamma_\alpha).
\end{equation}

\begin{cor}
Let $1 \neq \epsilon \in \mathbb{C}$ and $f \in \dot{H}^{1/2}(\Gamma_\alpha)$. Then the transmission problem 
$$\begin{cases} 
\int_{\R^3} |\nabla U|^2 \, dV < \infty, \\
\Delta U = 0 \textrm{ in } \Gamma_{\alpha, +}\cup\Gamma_{\alpha, -}, \\
\Tr_+ U - \Tr_- U = f \in \dot{H}^{1/2}(\Gamma_\alpha), \\
\partial_n^+ U - \epsilon \partial_n^- U = g \in \dot{H}^{-1/2}(\Gamma_{\alpha})
\end{cases}$$
is well posed (modulo constants) for all $g \in \E(\Gamma_{\alpha}) \simeq  \dot{H}^{-1/2}(\Gamma_{\alpha})$ if and only if 
\begin{equation} \label{eq:Esolvcond}
\frac{1+\epsilon}{1-\epsilon} \notin [-|1 - \alpha/\pi|, |1 - \alpha/\pi|].
\end{equation}
\end{cor}
\begin{proof}
By Corollary~\ref{cor:dirichletE} there are densities $h_\pm \in \E(\Gamma_{\alpha})$ and a constant $c$ such that $U = \Si_{\alpha} h_+ + c$ in $\Gamma_{\alpha, +}$ and $U = \Si_{\alpha} h_- + c$ in $\Gamma_{\alpha, -}$. By the jump formulas \eqref{eq:Ejumpformulas}, the transmission problem is then equivalent to the system
$$\begin{cases}
\Si_\alpha(h_+ - h_-) = f, \\
\left(\K_{\alpha} - \frac{1+\epsilon}{1-\epsilon}\mathbb{I}\right) h_+ = -\frac{1}{1-\epsilon} \left[2g + \epsilon(\K_\alpha+\mathbb{I})(h_+ - h_-) \right]
\end{cases}$$
on $\Gamma_{\alpha}$, where $\mathbb{I}$ denotes the identity map. This system is uniquely solvable if and only if \eqref{eq:Esolvcond} holds, by Corollary~\ref{cor:Sbdaryisom} and Theorem~\ref{thm:main2}.
\end{proof}

\bibliographystyle{amsplain}
\bibliography{wedgespec}

\providecommand{\bysame}{\leavevmode\hbox to3em{\hrulefill}\thinspace}
\providecommand{\MR}{\relax\ifhmode\unskip\space\fi MR }
\providecommand{\MRhref}[2]{%
  \href{http://www.ams.org/mathscinet-getitem?mr=#1}{#2}
}
\providecommand{\href}[2]{#2}
\begin{thebibliography}{10}

\bibitem{AS72}
M.~Abramowitz and I.~A. Stegun, \emph{Handbook of mathematical functions with
  formulas, graphs, and mathematical tables}, 9th printing, Dover, New York,
  1972.

\bibitem{ASSE}
Andrea Al\`u, M\'ario~G. Silveirinha, Alessandro Salandrino, and Nader Engheta,
  \emph{Epsilon-near-zero metamaterials and electromagnetic sources: Tailoring
  the radiation phase pattern}, Phys. Rev. B \textbf{75} (2007), 155410.

\bibitem{AMRZ}
Habib Ammari, Pierre Millien, Matias Ruiz, and Hai Zhang, \emph{Mathematical
  analysis of plasmonic nanoparticles: the scalar case}, Arch. Ration. Mech.
  Anal. \textbf{224} (2017), no.~2, 597--658.

\bibitem{ARYZ}
Habib Ammari, Matias Ruiz, Sanghyeon Yu, and Hai Zhang, \emph{Mathematical
  analysis of plasmonic resonances for nanoparticles: the full {M}axwell
  equations}, J. Differential Equations \textbf{261} (2016), no.~6, 3615--3669.

\bibitem{BCD11}
Hajer Bahouri, Jean-Yves Chemin, and Rapha\"el Danchin, \emph{Fourier analysis
  and nonlinear partial differential equations}, Grundlehren der Mathematischen
  Wissenschaften, vol. 343, Springer, Heidelberg, 2011.

\bibitem{BZ17}
Eric Bonnetier and Hai Zhang, \emph{Characterization of the essential spectrum
  of the {N}eumann--{P}oincar\'e operator in 2{D} domains with corner via
  {W}eyl sequences}, to appear in Rev. Mat. Iberoam.

\bibitem{CHM15}
S.~N. Chandler-Wilde, D.~P. Hewett, and A.~Moiola, \emph{Interpolation of
  {H}ilbert and {S}obolev spaces: quantitative estimates and counterexamples},
  Mathematika \textbf{61} (2015), no.~2, 414--443.

\bibitem{Cost07}
Martin Costabel, \emph{Some historical remarks on the positivity of boundary
  integral operators}, Boundary element analysis, Lect. Notes Appl. Comput.
  Mech., vol.~29, Springer, Berlin, 2007, pp.~1--27.

\bibitem{CM85}
Martin Costabel and Ernst Stephan, \emph{A direct boundary integral equation
  method for transmission problems}, J. Math. Anal. Appl. \textbf{106} (1985),
  no.~2, 367--413.

\bibitem{DK87}
Bj\"orn E.~J. Dahlberg and Carlos~E. Kenig, \emph{Hardy spaces and the
  {N}eumann problem in {$L^p$} for {L}aplace's equation in {L}ipschitz
  domains}, Ann. of Math. (2) \textbf{125} (1987), no.~3, 437--465.

\bibitem{NEPV12}
Eleonora Di~Nezza, Giampiero Palatucci, and Enrico Valdinoci,
  \emph{Hitchhiker's guide to the fractional {S}obolev spaces}, Bull. Sci.
  Math. \textbf{136} (2012), no.~5, 521--573.

\bibitem{Din96}
Z.~Ding, \emph{A proof of the trace theorem of {S}obolev spaces on {L}ipschitz
  domains}, Proc. Amer. Math. Soc. \textbf{124} (1996), no.~2, 591--600.

\bibitem{DM72}
L.~Dobrzynski and A.~A. Maradudin, \emph{Electrostatic edge modes in a
  dielectric wedge}, Phys. Rev. B \textbf{6} (1972), 3810--3815.

\bibitem{DM76}
M.~Duflo and Calvin~C. Moore, \emph{On the regular representation of a
  nonunimodular locally compact group}, J. Functional Analysis \textbf{21}
  (1976), no.~2, 209--243.

\bibitem{EL12}
Amit Einav and Michael Loss, \emph{Sharp trace inequalities for fractional
  {L}aplacians}, Proc. Amer. Math. Soc. \textbf{140} (2012), no.~12,
  4209--4216.

\bibitem{Els87}
Johannes Elschner, \emph{Asymptotics of solutions to pseudodifferential
  equations of {M}ellin type}, Math. Nachr. \textbf{130} (1987), no.~1,
  267--305.

\bibitem{EFV92}
L.~Escauriaza, E.~B. Fabes, and G.~Verchota, \emph{On a regularity theorem for
  weak solutions to transmission problems with internal {L}ipschitz
  boundaries}, Proc. Amer. Math. Soc. \textbf{115} (1992), no.~4, 1069--1076.

\bibitem{EM04}
Luis Escauriaza and Marius Mitrea, \emph{Transmission problems and spectral
  theory for singular integral operators on {L}ipschitz domains}, J. Funct.
  Anal. \textbf{216} (2004), no.~1, 141--171.

\bibitem{ET79}
Pierre Eymard and Marianne Terp, \emph{La transformation de {F}ourier et son
  inverse sur le groupe des {$ax+b$}\ d'un corps local}, Analyse harmonique sur
  les groupes de {L}ie ({S}\'em., {N}ancy-{S}trasbourg 1976--1978), {II},
  Lecture Notes in Math., vol. 739, pp.~207--248.

\bibitem{FJR78}
E.~B. Fabes, M.~Jodeit, Jr., and N.~M. Rivi\`ere, \emph{Potential techniques
  for boundary value problems on {$C^{1}$}-domains}, Acta Math. \textbf{141}
  (1978), no.~3-4, 165--186.

\bibitem{FJL77}
E.~B. Fabes, Max Jodeit, Jr., and Jeff~E. Lewis, \emph{Double layer potentials
  for domains with corners and edges}, Indiana Univ. Math. J. \textbf{26}
  (1977), no.~1, 95--114.

\bibitem{FMM98}
Eugene Fabes, Osvaldo Mendez, and Marius Mitrea, \emph{Boundary layers on
  {S}obolev-{B}esov spaces and {P}oisson's equation for the {L}aplacian in
  {L}ipschitz domains}, J. Funct. Anal. \textbf{159} (1998), no.~2, 323--368.

\bibitem{Fuhr06}
Hartmut F\"uhr, \emph{Hausdorff-{Y}oung inequalities for group extensions},
  Canad. Math. Bull. \textbf{49} (2006), no.~4, 549--559.

\bibitem{GM47}
I.~Gelfand and M.~Neumark, \emph{Unitary representations of the group of linear
  transformations of the straight line}, C. R. (Doklady) Acad. Sci URSS (N.S.)
  \textbf{55} (1947), 567--570.

\bibitem{GM90}
N.~V. Grachev and V.~G. Maz'ya, \emph{The {F}redholm radius of integral
  operators of potential theory}, Nonlinear equations and variational
  inequalities. {L}inear operators and spectral theory ({R}ussian), Probl. Mat.
  Anal., vol.~11, 1990, (Translated in J. Soviet Math. {{\bf{6}}4} (1993), no.
  6, 1297--1313), pp.~109--133, 251.

\bibitem{HSS05}
S.~Hassi, Z.~Sebesty\'en, and H.~S.~V. de~Snoo, \emph{On the nonnegativity of
  operator products}, Acta Math. Hungar. \textbf{109} (2005), no.~1-2, 1--14.

\bibitem{HP13}
Johan Helsing and Karl-Mikael Perfekt, \emph{On the polarizability and
  capacitance of the cube}, Appl. Comput. Harmon. Anal. \textbf{34} (2013),
  no.~3, 445--468.

\bibitem{HP18}
\bysame, \emph{The spectra of harmonic layer potential operators on domains
  with rotationally symmetric conical points}, J. Math. Pures Appl. (2018),
  Article in Press.

\bibitem{HMT10}
Steve Hofmann, Marius Mitrea, and Michael Taylor, \emph{Singular integrals and
  elliptic boundary problems on regular {S}emmes-{K}enig-{T}oro domains}, Int.
  Math. Res. Not. IMRN (2010), no.~14, 2567--2865.

\bibitem{KLY15}
Hyeonbae Kang, Mikyoung Lim, and Sanghyeon Yu, \emph{Spectral resolution of the
  {N}eumann-{P}oincar\'e operator on intersecting disks and analysis of plasmon
  resonance}, Arch. Ration. Mech. Anal. \textbf{226} (2017), no.~1, 83--115.

\bibitem{Ken85}
Carlos~E. Kenig, \emph{Recent progress on boundary value problems on
  {L}ipschitz domains}, Pseudodifferential operators and applications ({N}otre
  {D}ame, {I}nd., 1984), Proc. Sympos. Pure Math., vol.~43, Amer. Math. Soc.,
  Providence, RI, 1985, pp.~175--205.

\bibitem{Khal74}
Idriss Khalil, \emph{Sur l'analyse harmonique du groupe affine de la droite},
  Studia Math. \textbf{51} (1974), 139--167.

\bibitem{KPS07}
Dmitry Khavinson, Mihai Putinar, and Harold~S. Shapiro, \emph{Poincar\'e's
  variational problem in potential theory}, Arch. Ration. Mech. Anal.
  \textbf{185} (2007), no.~1, 143--184.

\bibitem{KR78}
Abel Klein and Bernard Russo, \emph{Sharp inequalities for {W}eyl operators and
  {H}eisenberg groups}, Math. Ann. \textbf{235} (1978), no.~2, 175--194.

\bibitem{KL72}
Adam Kleppner and Ronald~L. Lipsman, \emph{The {P}lancherel formula for group
  extensions. {I}, {II}}, Ann. Sci. \'Ecole Norm. Sup. (4) \textbf{5} (1972),
  459--516; ibid. (4) 6 (1973), 103--132.

\bibitem{KMR97}
V.~A. Kozlov, V.~G. Maz'ya, and J.~Rossmann, \emph{Elliptic boundary value
  problems in domains with point singularities}, Mathematical Surveys and
  Monographs, vol.~52, American Mathematical Society, Providence, RI, 1997.

\bibitem{Kre98}
M.~G. Krein, \emph{Compact linear operators on functional spaces with two
  norms}, Integral Equations Operator Theory \textbf{30} (1998), no.~2,
  140--162, Translated from the Ukranian, Dedicated to the memory of Mark
  Grigorievich Krein (1907--1989).

\bibitem{Lew90}
Jeff~E. Lewis, \emph{Layer potentials for elastostatics and hydrostatics in
  curvilinear polygonal domains}, Trans. Amer. Math. Soc. \textbf{320} (1990),
  no.~1, 53--76.

\bibitem{Lew91}
\bysame, \emph{A symbolic calculus for layer potentials on {$C^1$} curves and
  {$C^1$} curvilinear polygons}, Proc. Amer. Math. Soc. \textbf{112} (1991),
  no.~2, 419--427.

\bibitem{LP83}
Jeff~E. Lewis and Cesare Parenti, \emph{Pseudodifferential operators of
  {M}ellin type}, Comm. Partial Differential Equations \textbf{8} (1983),
  no.~5, 477--544.

\bibitem{LM72}
J.-L. Lions and E.~Magenes, \emph{Non-homogeneous boundary value problems and
  applications. {V}ol. {I}}, Springer-Verlag, New York-Heidelberg, 1972,
  Translated from the French by P. Kenneth, Die Grundlehren der mathematischen
  Wissenschaften, Band 181.

\bibitem{Garnett}
J.C Maxwell~Garnett, \emph{{VII}. {C}olours in metal glasses, in metallic
  films, and in metallic solutions.{\textemdash}{II}}, Philos. Trans. Royal
  Soc. A \textbf{205} (1906), no.~387-401, 237--288.

\bibitem{McCar92}
John~E. McCarthy, \emph{Geometric interpolation between {H}ilbert spaces}, Ark.
  Mat. \textbf{30} (1992), no.~2, 321--330.

\bibitem{Medk18}
Dagmar Medkov\'a, \emph{The {L}aplace equation: {B}oundary value problems on
  bounded and unbounded {L}ipschitz domains}, Springer International
  Publishing, 2018.

\bibitem{Mil02}
G.~W. Milton, \emph{The theory of composites}, Cambridge University Press,
  Cambridge, UK, 2002.

\bibitem{Mit02}
Irina Mitrea, \emph{On the spectra of elastostatic and hydrostatic layer
  potentials on curvilinear polygons}, J. Fourier Anal. Appl. \textbf{8}
  (2002), no.~5, 443--487.

\bibitem{MW12}
Marius Mitrea and Matthew Wright, \emph{Boundary value problems for the
  {S}tokes system in arbitrary {L}ipschitz domains}, Ast\'erisque (2012),
  no.~344, viii+241.

\bibitem{NMM93}
N.~A. Nicorovici, R.C McPhedran, and G.~W. Milton, \emph{Transport properties
  of a three-phase composite material: the square array of coated cylinders},
  Proc. R. Soc. Lond. A \textbf{442} (1993), no.~1916, 599--620.

\bibitem{NL95}
K.~I. Nikoshkinen and I.~V. Lindell, \emph{Image solution for {P}oisson's
  equation in wedge geometry}, IEEE. T. Antenn. Propag. \textbf{43} (1995),
  no.~2, 179--187.

\bibitem{PP14}
Karl-Mikael Perfekt and Mihai Putinar, \emph{Spectral bounds for the
  {N}eumann-{P}oincar\'e operator on planar domains with corners}, J. Anal.
  Math. \textbf{124} (2014), no.~1, 39--57.

\bibitem{PP16}
\bysame, \emph{The {E}ssential {S}pectrum of the {N}eumann--{P}oincar\'e
  {O}perator on a {D}omain with {C}orners}, Arch. Ration. Mech. Anal.
  \textbf{223} (2017), no.~2, 1019--1033.

\bibitem{Qiao18}
Yu~Qiao, \emph{Double layer potentials on three-dimensional wedges and
  pseudodifferential operators on {L}ie groupoids}, J. Appl. Math. Anal. Appl.
  \textbf{462} (2018), no.~1, 428--447.

\bibitem{NQ12}
Yu~Qiao and Victor Nistor, \emph{Single and double layer potentials on domains
  with conical points {I}: {S}traight cones}, Integral Equations Operator
  Theory \textbf{72} (2012), no.~3, 419--448.

\bibitem{Schar04}
R.~W. Scharstein, \emph{Green's function for the harmonic potential of the
  three-dimensional wedge transmission problem}, IEEE. T. Antenn. Propag.
  \textbf{52} (2004), no.~2, 452--460.

\bibitem{Schm12}
Konrad Schm\"udgen, \emph{Unbounded self-adjoint operators on {H}ilbert space},
  Graduate Texts in Mathematics, vol. 265, Springer, Dordrecht, 2012.

\bibitem{Shele90}
V.~Yu. Shelepov, \emph{On the index and spectrum of integral operators of
  potential type along {R}adon curves}, Mat. Sb. \textbf{181} (1990), no.~6,
  751--778.

\bibitem{VS13}
C.~A. Valagiannopoulos and A.~Sihvola, \emph{Improving the electrostatic field
  concentration in a negative-permittivity wedge with a grounded "bowtie"
  configuration}, Radio Sci. \textbf{48} (2013), no.~3, 316--325.

\bibitem{Ver84}
Gregory Verchota, \emph{Layer potentials and regularity for the {D}irichlet
  problem for {L}aplace's equation in {L}ipschitz domains}, J. Funct. Anal.
  \textbf{59} (1984), no.~3, 572--611.

\bibitem{WKS08}
Henrik Wall\'en, Henrik Kettunen, and Ari Sihvola, \emph{Surface modes of
  negative-parameter interfaces and the importance of rounding sharp corners},
  Metamaterials \textbf{2} (2008), 113--121.

\bibitem{YA18}
S.~Yu and H.~Ammari, \emph{Plasmonic interaction between nanospheres}, SIAM
  Rev. \textbf{60} (2018), no.~2, 356--385.

\end{thebibliography}

\end{document}